\documentclass{article}
\usepackage{graphicx}
\usepackage[a4paper, total={6in, 8in}]{geometry}

\usepackage{amssymb}
\usepackage[fleqn]{amsmath}
\usepackage{multirow}
\usepackage{pdflscape}
\usepackage{rotating}
\usepackage{tikz}
\usepackage{graphicx}
\usepackage{tikz-qtree}
\usetikzlibrary{trees,shapes,arrows} 
\usetikzlibrary{shapes,arrows}
\usetikzlibrary{calc}
\usepackage{booktabs}      
\usepackage{adjustbox}     
\usepackage{rotating}
\usepackage{pgfplots}
\usepackage{amsfonts,amsthm,epsfig,epstopdf,titling,url,array}
\usepackage{thmtools,thm-restate}
\usepackage{afterpage}
\usepackage{floatpag}
\usepackage{tablefootnote}
\usepackage{threeparttable}
\usepackage{multicol}
\usepackage{longtable}

\theoremstyle{definition}
\newtheorem{defn}{Definition}[section]

\newtheorem{exmp}{Example}[section]

\newtheorem{theorem}{Theorem}[section]
\newtheorem{corollary}{Corollary}[theorem]

\newcommand{\contradiction}{%
\begin{tikzpicture}[rotate=45,x=0.5ex,y=0.5ex]
\draw[line width=.2ex] (0,2) -- (3,2) (0,1) -- (3,1) (1,3) -- (1,0) (2,3) -- (2,0);
\end{tikzpicture}
}

\usepackage{graphicx}
\definecolor{eafbest}{HTML}{FDE725}
\definecolor{eafworst}{HTML}{440154}
\newcommand\cbox[1]{\textcolor{#1}{\rule{0.7em}{0.7em}}}
\usepackage[normalem]{ulem} 

\def\genbox#1#2#3#4#5#6{
    \leavevmode\raise#4bp\hbox to#5bp{\vrule height#5bp depth0bp width0bp
    \pdfliteral{q .5 w \csname #2COLOR\endcsname\space RG
                       \csname #3PDF\endcsname{#5}{#6} S Q
             \ifx1#1 q \csname #2COLOR\endcsname\space rg 
                       \csname #3PDF\endcsname{#5}{#6} f Q\fi}\hss}}

\usepackage{authblk}

\title{A novel multi-stage multi-scenario multi-objective optimisation framework for adaptive robust decision-making under deep uncertainty}

\author[1,2]{Babooshka Shavazipour\footnote{Corresponding author\\
Email addresses: babooshka.b.shavazipour@jyu.fi, b.shavazipour@gmail.com (Babooshka Shavazipour)}} 

\author[2]{Theodor J.Stewart}

\affil[1]{University of Jyvaskyla, Faculty of Information Technology, P.O. Box 35 (Agora), FI-40014 University of Jyvaskyla, Finland}
\affil[2]{Department of Statistical Sciences, University of Cape Town, Rondebosch, Cape Town, South Africa}

\usepackage{natbib}
\usepackage{graphicx}
 
\date{}

\begin{document}
\maketitle
\begin{abstract} 
Many real-world decision-making problems involve multiple decision-making stages and various objectives. Besides, most of the decisions need to be made before having complete knowledge about all aspects of the problem leaves some sort of uncertainty. Deep uncertainty happens when the degree of uncertainty is so high that the probability distributions are not  confidently knowable. In this situation, using wrong probability distributions lead to failure. Scenarios, instead, should be used to evaluate the consequences of any decisions in different plausible futures and find a robust solution. In this study, we proposed a novel multi-stage multi-scenario multi-objective optimisation framework for adaptive robust decision-making under deep uncertainty. Two approaches, named multi-stage multi-scenario multi-objective and two-stage moving horizon, have been proposed and compared. 
Finally, the proposed approaches are applied in a case study of sequential portfolio selection under deep uncertainty and the robustness of their solutions is discussed.

    
\end{abstract}

\section{Introduction}
  Sustainable decision-making requires simultaneous consideration of multiple conflicting objectives, e.g., economics versus environmental and social objectives. Typically, there is no single optimal solution satisfying all the conflicting objectives at the same time. Instead, there are several compromises called \emph{Pareto optimal} with different trade-offs. Therefore, the decision-maker must find the most balanced Pareto optimal solution based on their preferences. Because the Pareto optimal solutions are unknown a priori, additional support is needed to generate those solutions (e.g., utilising mathematical models) and help the decision-maker in trade-off analyses and decision-making. \emph{Multi-objective optimisation} methods have been designed to provide this kind of support to decision-makers (see, e.g., \cite{hwang1979, Steuer1986, miettinen1999} and references therein). They also applied in many application areas (e.g., \cite{Abdelaziz2001, Abdelaziz2009, Masri2016, Abdelazizn2017, Sindhya2017, Eyvindson2018a, Montonen2019, AghaeiPour2021, Shavazipour2022KneeOA, Shavazipour2022forest, saini2023interactive}). 
 
 Moreover, many real-life decision problems are tied to unpredictable events related to future and human behaviours, such as climate change, natural hazards, socioeconomic, pandemics, military/ cyber-attacks, political crises, and religious beliefs. Decision-makers need to make decisions with incomplete knowledge about the consequences, all available options, and the future state of the world (referred to as \emph{scenarios} from now on) \citep{VanderHeijden1996, Shavazipour2019}. Because of different sources of uncertainties and the lack of historical data, precise forecasting of future outcomes and transitions is impossible. In this case, the probabilities of plausible future scenarios are unknown, or various experts cannot agree upon one, classified as \emph{deep uncertainty} \citep{Bankes2002, Lempert2003, Walker2013, Shavazipour2019}. Also, the outcomes and performances of any decision may vary in different scenarios, making the decision-making process too convoluted \citep{Shavazipour2021b}. 
 The class of optimisation problems considers multiple objectives under deep uncertainty, known as \emph{multi-scenario} (or scenario-based) \emph{multi-objective optimisation} \citep{Watson2017, Eker2018, Shavazipour2019, Shavazipour2021a}. \emph{Scenarios}, in this context, are utilised to structure and organise the present uncertainty by defining different plausible scenarios for uncertain statements of the problem. Using scenarios to represent uncertainty in a problem can allow the decision-maker to think, discuss, and plan for various plausible representations of an uncertain event \citep{Durbach2012, Stewart2013} without considering the probabilities of their occurrences.

In contrast to a wide variety of methods developed for stochastic and probabilistic multi-objective optimisation problems (see, e.g., \cite{Abdelaziz2007, Abdelazizn2017, Masri2016}), classified as \emph{mild} uncertainty \citep{Shavazipour2019}, only a few studies can be found in the literature investigating multi-objective optimisation under \emph{deep} uncertainty (e.g., \cite{Kasprzyk2013, Quinn2017, Watson2017, Eker2018, Shavazipour2019, Bartholomew2020, Shavazipour2021a}), despite the need. Nonetheless, handling deep uncertainty 
has turned into a hot topic in various fields, particularly after worldwide crises in the third decade of the twenty-first century---i.e., the COVID-19 pandemic and the Russian invasion of Ukraine that led to economic and energy crises. 

The most recommended approach to cope with deep uncertainty is to \emph{monitor and adapt} \citep{marchau2019decision}, meaning that the system should always be ready to be adapted if a particular scenario manifests itself. Indeed, one needs to identify various scenario-specific adaptation (or contingency) plans in advance and implement the right one after scenario realisation. This way, one can avoid over-conservatism by implementing the relevant adaptation plan after scenario realisation. 
\cite{Shavazipour2019} proposed a two-stage multi-scenario, multi-objective optimisation structure to handle deep uncertainty and identify dynamic-robust Pareto optimal solutions. 
In the two-stage structure, the decision-making process has been divided into two stages: In the first stage, knowledge about some states of the problem is lacking, and this is when the \emph{initial decision} needs to be made and implemented. The second stage occurs after unfolding the uncertain states when the adaptation (or recourse) decisions can be implemented. Therefore, a more reasonable initial decision, which is common among all plausible scenarios, and a set of \emph{adaptation decisions} relating to different scenarios, one for every plausible scenario, is made. However, only the initial decision will be implemented in the first stage, while the implementation of the adaptation decision is postponed to the second stage after scenario realisation. Then, the relevant adaptation decision, which depends on the unfolded scenario, would be implemented. 
This approach has also been successfully applied in a South-African sugar industry case study \citep{Shavazipour2020}.

However, real processes are continuous 
and scenarios continuously unfold, generally with some dependency. Hence, decisions (or plans) must be frequently adapted in more than two stages. Note that, in this sense, \emph{``stages"} are artificial to structure thinking of an \emph{indefinite future}. 
Furthermore, although the two-stage structure lets us evaluate the consequences of the initial decision after achieving every plausible scenario, the long-term effects of the initial decision and the outcomes of the adaptation decisions need to be investigated in a more extended structure. 

Therefore, this paper aims to address the above-mentioned gaps and contributes to the multi-objective robust decision-making by proposing a \emph{multi-stage multi-scenario multi-objective robust optimisation (MS2MORO) framework} for dynamic-robust decision-making. 
To develop a robust strategy for an indefinite future, one may consider the following options: (1) look multiple steps ahead---always plan as if it is a \emph{T-stage horizon}; or (2) look one step ahead---always plan as if it is a two-stage planning window and then start another two-stage planning window after the first scenario realisation stage (\emph{two-stage moving horizon}).
Clearly, considering more stages adds more uncertainty and complexity to the problem, which naturally requires additional computation resources. 
Indeed, it introduces a new trade-off between getting better performances and computation resources. 
Accordingly, in this paper, we also propose a two-stage moving horizon approach (within a T-stage structure) and compare it with the other proposed multi-stage approach. 
Finally, we apply both proposed approaches to a sequential portfolio selection example and compare the results.
 
This paper is organised as follows: 
Section \ref{back} briefly reviewed necessary concepts and notations. 
The proposed multi-stage multi-scenario multi-objective optimisation framework, relevant concepts and models, and the proposed moving horizon approach are discussed in Section \ref{ms2moo}. In Section \ref{portfolio}, all the proposed structures and approaches are illustrated and compared in detail through a sequential portfolio selection application under deep uncertainty. Further robustness analysis and discussion are performed in Section \ref{dis}, before we conclude in Section \ref{3sconclusion}.

\section{Background}\label{back}

\subsection{Multi-scenario Multi-objective Optimisation}\label{msmop}
A multi-scenario multi-objective optimisation problem under deep uncertainty can be formulated as follows \citep{Shavazipour2021b}:
 
\begin{equation}
\begin{array}{rll}
\mbox{Min}  & {\{f_{1p}({\mathbf{x}}),  \dots, f_{mp}({\mathbf{x}}) \}}, & p \in \Omega \\ 
\mbox{subject to}  & {\mathbf{x}} \in {\mathbf{X}}, & \\
\end{array}\label{eq:sbmop}
\end{equation}
where $m$ $(\geq 2)$ is the number of objective functions, the scenario space represents by $\Omega = \{1, \dots, s\}$; $s$ is the number of scenarios;
$f_{ip} $ demonstrates an objective function $i$ in a scenario $p$; ${\mathbf{X}}$, in the \emph{decision space} $\Re^n $, includes all feasible solutions, each represented by a vector of decision variables $ \mathbf{x} =  (x_1, \dots, x_n)^T $. An \emph{objective vector} $\mathbf{z}_p =(f_{1p}({\mathbf{x}}),  \dots, f_{mp}({\mathbf{x}}))^T$ is the image of a solution $\mathbf{x}$ in a scenario $p$ , and the \emph{objective space} $\Re^k$ is constructed by the set of all objective vectors.

A decision vector (feasible solution) $\mathbf{x}^*$ is Pareto optimal with respect to the scenario set $\Omega$, if there does not exist any other feasible decision vector $\hat{\mathbf{x}} \in \mathbf{X}$ with better values in one objective function in one scenario without impairing any objective function in any scenario. Mathematically speaking $\forall i,p, f_{ip}(\mathbf{x}^*) \leq f_{ip}(\hat{\mathbf{x}})$ and $\exists j,q, s.t., f_{jq}(\mathbf{x}^*) < f_{jq}(\hat{\mathbf{x}})$---i.e., dominated at least in one objective in one scenario.

In such a complex decision-making problem, a decision-maker requires to compare the trade-offs between objectives in multiple scenarios as well as the robustness of the Pareto optimal solutions and choose the most preferred solution based on their preferences. In the multi-objective optimisation literature (without uncertainty consideration), different types of preferences have been introduced. One of the most common preference types is called a \emph{reference point} which consists of \emph{aspiration levels} representing desired objective values to the decision-maker. In multi-scenario problems, the aspiration levels can be extended to the desired values of objective functions in various scenarios (see, e.g., \citet{Shavazipour2021a, Shavazipour2022forest}).

\subsection{Goal programming (the reference point method)}\label{GP}
Different approaches have been developed to solve a multi-objective optimisation problem in the literature. One way is to transform a multi-objective optimisation into a single-objective problem (often using the so-called scalarization function) 
and then solve that single-objective problem with a suitable single-objective solver. One of the most popular \emph{multi-criteria decision-making} (MCDM) techniques using the single-objective transformation idea is the goal programming (GP) approach \citep{Charnes1955, charnes1961}. Among different variants, the \emph{reference point method} (RPM), is shown to always directly generate Pareto optimal solutions \citep{Ogryczak1994}. The reference point method (also known as generalised goal programming) can be formulated as the following:

\begin{equation}\label{eq:RGP}
\begin{array}{l l l}
Min \; & \max\limits_{i} \{ \omega_{i} (f_i(\mathbf{x}) - g_i) \} + \epsilon \sum\limits_{i=1}^m \omega_{i} (f_i(\mathbf{x}) - g_i) & \\
 s.t.   & \mathbf{x} \in \mathbf{X} & \\
 \end{array}
\end{equation}
where $\omega_{i}>0$ $(i=1, ..., m)$ are the importance weighting of deviations from the goals set by the decision-maker.
The reference point and its desired objective values (aspiration levels) are represented by $\mathbf{g}$ and $g_{i}$, respectively. $\epsilon$ is an arbitrarily small, positive number to guarantee the Pareto efficiency of the solutions. \textbf{x} is a vector of decision variables, and \textbf{X} is the set of feasible solutions. 

When we have multiple scenarios, the above formulation has an additional dimension and then a multi-scenario form of the reference point goal programming model can be formulated as follows:

\begin{equation}\label{eq:sbgp}
\begin{array}{lll}
   Min  & \max\limits_{i , p} \{ \omega_{ip}(f_{ip}(\mathbf{x}) - g_{ip}) \} + \epsilon \sum\limits_{i=1}^m \sum\limits_{p=1}^s  \omega_{ip}(f_{ip}(\mathbf{x}) - g_{ip}) & \\
     s.t. & \mathbf{x} \in \mathbf{X}, & \\
\end{array}
\end{equation}%
where
weights $\omega_{ip}>0$ represents the importance of deviation from the goals set for the objective $i$ in scenario $p$.
The desired objective values (aspiration levels) in different scenarios are represented by $g_{ip}$, respectively.

\section{Proposed multi-stage multi-scenario multi-objective optimisation framework}\label{ms2moo}

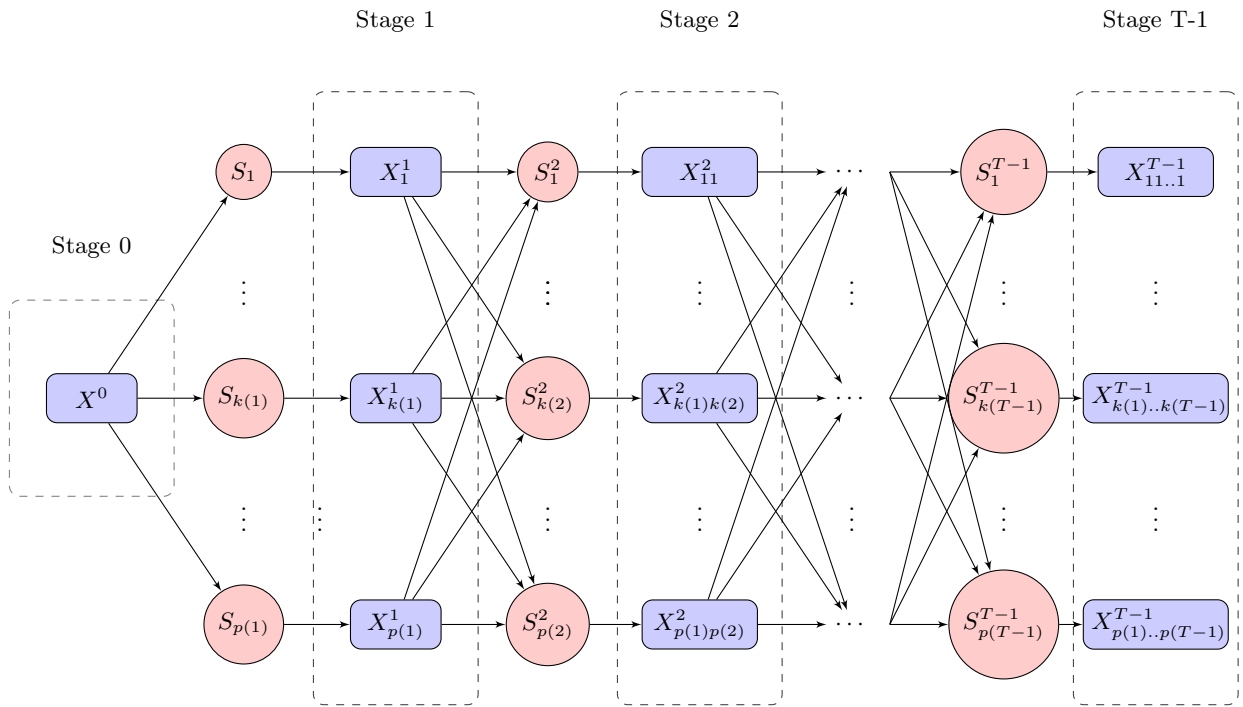
\begin{figure}[!htbp]
\tikzstyle{decision} = [diamond, draw, fill=blue!20, 
    text width=5.5em, text badly centered, node distance=3cm, inner sep=0pt]
\tikzstyle{block} = [rectangle, draw, fill=blue!20, 
    text width=3em, text centered, rounded corners, minimum height=2em]
    \tikzstyle{block2} = [rectangle, draw, fill=blue!20, 
    text width=4em, text centered, rounded corners, minimum height=2em]
    \tikzstyle{block2w} = [rectangle, draw, fill=blue!20, 
    text width=5.2em, text centered, rounded corners, minimum height=2em]
\tikzstyle{block3} = [rectangle, draw, dashed, gray,
    text width=6em, text centered, rounded corners, minimum height=8em]
\tikzstyle{block4} = [rectangle, draw, dashed, darkgray,
    text width=6em, text centered, rounded corners, minimum height=25em]

\tikzstyle{line} = [draw, -latex']
\tikzstyle{cloud} = [draw, circle,fill=red!20, node distance=1.5cm,
    minimum height=1.5em]
\tikzstyle{emt} =[node distance=1cm]

\begin{center}
   \small 
\begin{tikzpicture}[node distance = 1cm, auto]

    \node [block] (x0) at (0,3) {$X^0$};
     \node [block3] (x00) at (0,3) {};    
     \node (x000) at (0,5) {Stage 0};   
    
    \node [cloud] (s1) at (2,6) {$S_1$};
        \node (s10) at (2,4.5) {$\vdots$}; 
    \node [cloud] (sk1) at (2,3) {$S_{k(1)}$};
    \node (s100) at (2,1.5) {$\vdots$};
    \node [cloud] (sp1) at (2,0) {$S_{p(1)}$};
     
    \node [block] (x1) at (4,6) {$X_1^1$};
    \node (s11) at (6,4.5) {$\vdots$};
    \node [block] (xk1) at (4,3) {$X_{k(1)}^1$};
    \node (s12) at (3,1.5) {$\vdots$};
    \node [block] (xp1) at (4,0) {$X_{p(1)}^1$};
    \node [block4] (x10) at (4,3) {}; 
    \node (x100) at (4,8) {Stage 1}; 
    
    \node [cloud] (s2) at (6,6) {$S_1^{2}$};
    \node (s21) at (6,4.5) {$\vdots$};
    \node [cloud] (sk2) at (6,3) {$S_{k(2)}^{2}$};
    \node (s22) at (6,1.5) {$\vdots$};
    \node [cloud] (sp2) at (6,0) {$S_{p(2)}^{2}$};
    
    \node [block2] (x2) at (8,6) {$X_{11}^2$};
    \node (s31) at (8,4.5) {$\vdots$};
    \node [block2] (xk2) at (8,3) {$X_{k(1)k(2)}^2$};
    \node (s32) at (8,1.5) {$\vdots$};
    \node [block2] (xp2) at (8,0) {$X_{p(1)p(2)}^2$};
        \node [block4] (x20) at (8,3) {}; 
          \node (x200) at (8,8) {Stage 2}; 

     \node [] (s3) at (10,6) {$\cdots$};
    \node (s31) at (10,4.5) {$\vdots$};
    \node [] (sk3) at (10,3) {$\cdots$};
    \node (s42) at (10,1.5) {$\vdots$};
    \node [] (sp3) at (10,0) {$\cdots$};

    \node [cloud] (st) at (12,6) {$S_1^{T-1}$};
    \node (s31) at (12,4.5) {$\vdots$};
    \node [cloud] (skt) at (12,3) {$S_{k(T-1)}^{T-1}$};
    \node (s52) at (12,1.5) {$\vdots$};
    \node [cloud] (spt) at (12,0) {$S_{p(T-1)}^{T-1}$};

    \node [block2] (xt) at (14,6) {$X_{11..1}^{T-1}$};
    \node (s31) at (14,4.5) {$\vdots$};
    \node [block2w] (xkt) at (14,3) {$X_{k(1)..k(T-1)}^{T-1}$};
    \node (s62) at (14,1.5) {$\vdots$};
    \node [block2w] (xpt) at (14,0) {$X_{p(1)..p(T-1)}^{T-1}$};
        \node [block4] (xt0) at (14,3) {}; 
          \node (xt00) at (14,8) {Stage T-1}; 

    \path [line] (x0) --   (s1); 
    \path [line] (x0) --   (sk1);   
    \path [line] (x0) --   (sp1);   
  
    \path [line] (s1) --  (x1);

    \path [line] (sk1) --  (xk1);

    \path [line] (sp1) --  (xp1);

     \path [line] (x1) --   (s2);
     \path [line] (x1) --   (sk2);   
     \path [line] (x1) --   (sp2);   
   
   \path [line] (xk1) --   (s2);
     \path [line] (xk1) --   (sk2);   
     \path [line] (xk1) --   (sp2);   
   
   \path [line] (xp1) --   (s2);
     \path [line] (xp1) --   (sk2);   
     \path [line] (xp1) --   (sp2);   
   
   \path [line] (s2) --  (x2);

    \path [line] (sk2) --  (xk2);

    \path [line] (sp2) --  (xp2);

     \path [line] (x2) --   (s3);
     \path [line] (x2) --   (sk3);   
     \path [line] (x2) --   (sp3);   
   
     \path [line] (xk2) --   (s3);
     \path [line] (xk2) --   (sk3);   
     \path [line] (xk2) --   (sp3);   
   
     \path [line] (xp2) --   (s3);
     \path [line] (xp2) --   (sk3);   
     \path [line] (xp2) --   (sp3);   
   
     \path [line] (10.5,6) --   (st);
     \path [line] (10.5,6) --   (skt);   
     \path [line] (10.5,6) --   (spt);   
   
     \path [line] (10.5,3) --   (st);
     \path [line] (10.5,3) --   (skt);   
     \path [line] (10.5,3) --   (spt);   

     \path [line] (10.5,0) --   (st);
     \path [line] (10.5,0) --   (skt);   
     \path [line] (10.5,0) --   (spt);   

    \path [line] (st) --  (xt);

    \path [line] (skt) --  (xkt);

    \path [line] (spt) --  (xpt);

\end{tikzpicture}
\caption{T-stage decision-making process with \textit{$p(1)\times p(2)\times ...\times p(T-1)$} scenarios.}\label{fig:Tstage} 
\end{center}
\end{figure}

In a T-stage structure (as shown in Figure \ref{fig:Tstage}), the decision process is divided into $T$ stages during the time horizon, which is arbitrarily chosen depending on the problem \footnote{Note that real processes are continuous, and by the time horizon, we try to structure and model a specific stage of an infinite future.}. These stages do not show some specific points in the time horizon and could be adapted to different conditions. We use stages to distinguish separate periods of the decision-making process and scenario realisation,---the time that an uncertain state of the problem is unfolded and the consequences of our previous decisions manifest themselves by then.

At the first stage (\texttt{stage 0}), an initial decision ($ x^0 \in \mathbf{X}^0 $, $\mathbf{X}^0$ is a feasible set of the initial decisions) must be made and implemented without any knowledge about future scenarios, a scenario-free decision. Then one needs to wait and see which scenarios will realise in the following stages in the future. Accordingly, suitable contingency decisions must be implemented at each stage. Indeed, one should identify a chain of scenario-dependant contingency decisions for each plausible scenario path in the following stages (to be prepared for the future). However, only one contingency decision set will be implemented after realising the future scenarios. The union of the initial decision ($ x^0$) and $T-1$ relevant contingency decisions ($ x_{k(1)}^1, x_{k(1)k(2)}^2, \cdots,  x_{k(1)k(2) ... k(T-1)}^{T-1}$) constitute the entire chain of decisions if scenarios $ s_{k(1)}, s_{k(2)}^2, \cdots,  s_{k(T-1)}^{T-1}$ are unfolded, respectively. 

In contrast to the multi-stage stochastic programming, the proposed T-stage structure requires no knowledge about scenarios' probabilities. Indeed, in this structure, scenarios represent plausible futures that are different from the standard form of decision trees with known scenario probabilities. Therefore, we can handle problems under deep uncertainty.
 
The conventional optimisation methods under uncertainty usually identify the optimal solution for the average (stochastic), the most probable (probabilistic), or the worst-case scenario (hard robustness). The first two approaches easily fail if a different scenario is realised, and the latter is too costly and inefficient. 
The proposed structure, however, looks for the optimal combination of a robust, adaptable initial decision (which is not irrecoverable in the future) and $T-1$ subsequent contingency decisions for all considered scenario paths in the proposed structure. 
This multi-stage decision-making process can be summarized as follows: 

\begin{itemize}
    \item[] \texttt{Stage 0}: \textbf{Initial decision} is made before any scenario realisation (scenario-free decision).
    \item[] \texttt{Stage 1}: \textbf{The first contingency decision} to be taken if scenario \textit{k(1)} reveals (scenario dependent decision).
    \item[] \texttt{Stage 2}: \textbf{The second contingency decision} to be taken if scenario \textit{k(1)} has been unfolded at \texttt{stage 1} and scenario \textit{k(2)} manifests in \texttt{stage 2} (scenario dependent decision).
    \item[]  $\vdots$
    \item[] \texttt{Stage (T-1)}: \textbf{The $(T-1)^{th}$ contingency decision} to be taken if scenario \textit{k(1)} has been unfolded at \texttt{stage 1}, scenario \textit{k(2)} manifests in \texttt{stage 2}, ..., and  scenario \textit{k(T-1)} manifests in \texttt{stage (T-1)}---scenario dependent decision. 
\end{itemize}

\begin{defn}\label{df1}
\textit{Meta-decision}: 
In the proposed T-stage structure, a chain (or series) of decisions ($\mathbf{X}$), including the joint initial decision ($\mathbf{x}^0$) and T-1 consecutive decisions ($\mathbf{x}_{k(1)}^1, \mathbf{x}_{k(1)k(2)}^2), \cdots, \mathbf{x}_{k(1)k(2)..k(T-1)}^{T-1})$) in a scenario path ($ s_{k(1)}, s_{k(2)}^2, \cdots,  s_{k(T-1)}^{T-1}$), is called a \emph{``meta-decision"}. 
\end{defn}

\subsection{Extension of the concept of meta-criterion in the T-stage structure}\label{metaextend}
 Before formulating the model, the concept of meta-criterion in a T-stage process should be discussed. \citet{Stewart2013} introduced the term \textit{``meta-criterion"}\footnote{It can be called \emph{meta-objective} in the optimization context.} (criterion-scenario combination) as a dimension of preferences in scenario-based  \textit{multi-criteria decision analysis} (MCDA)---i.e., the performance of each alternative (decision) should be evaluated and compared regarding each criterion under the conditions of each scenario. 
In the proposed T-stage structure, each \emph{meta-decision} is related to a scenario path---a combination of scenarios in $T-1$ scenario realisation stages. Therefore, an evaluation must be done to determine how well a \emph{meta-decision} performs concerning each criterion in each scenario path---i.e., each \emph{meta-criterion} represents preferences regarding a criterion in a scenario path.

\begin{defn}\label{df2}
\textit{Meta-objective}: 
A meta-objective is defined as the combination of an objective and a scenario path in which the number of all meta-objectives equals $m \times p$, where $m$ and $p$ represent the total number of objectives and scenario paths, respectively.  
\end{defn}

Therefore, scenario paths, as a component of meta-objectives, represent a valid dimension of preference that can be utilised to avoid problems of assessing and using probabilities under deep uncertainty.
We use the extended concept of meta-objectives to formulate a mathematical optimisation model to find Pareto optimal meta-decisions and expand the concept of \emph{robustness} and \emph{dynamic-robust decision} in a multi-stage decision-making structure.

\subsection{Mathematical Formulation}\label{mathformula}

In the proposed T-stage decision-making structure, a meta-decision splits into $T$ adaptable decisions and defers $T$ stages of the decision long enough to unfold the uncertainty gradually (one contingency decision after each stage of scenario realisation). A dynamic-robust decision in this structure is a meta-decision which contains $T$ subgroups of decisions: an initial decision ($X_0$) followed by $T-1$ contingency/recourse decisions ($X_{k(1),\cdots,k(T-1)}^{T-1}$), in which the meta-decision is robust across scenario paths. %

According to definitions \ref{df1} and \ref{df2}, in the proposed multi-stage multi-scenario multi-objective decision-making structure, there are $ m \times p$ meta-objectives\footnote{Without loss of generality, to simplify models and readability in all models, we assume $p=p(1) \times p(2) \times \cdots \times p(T-1)$. } that must be optimised, while some uncertainties will unfold after the decision-making step. Therefore, the model includes $ m \times p$ objective functions in which the performance of meta-decisions are compared and evaluated subject to the satisfaction of some conditions that could be stage-oriented. 

For simplicity, we assume the same number of objective functions in each scenario at each stage. Then, the proposed multi-stage multi-scenario multi-objective optimisation problem can be formulated as follows:

\begin{center}
\begin{equation}\label{eq:ms2mop}
 \begin{array}{l l}
Opt_{(\mathbf{x}^0, \mathbf{x}_{k(1)}^1, \mathbf{x}_{k(1)k(2)}^2, ..., \mathbf{x}_{k(1)k(2)...k(T-1)}^{T-1})} \, \mathbf{F} & k(t)=1,...,p(t);\\
&   t=1,..., T-1;\\
 & \\
 s.t. $ $ $ $  u_r^0(\mathbf{x}^0) \leq 0, & r=1, ...,R_0; \\
 & \\
 $ $ $ $ $ $ $ $ $ $ $ $  u_r^1(\mathbf{x}^0, \mathbf{x}_{k(1)}^1) \leq 0, & k(1)=1,..., p(1); \\ 
& r=1,..., R_1;  \\
 & \\
 $ $ $ $ $ $ $ $ $ $ $ $  u_r^t(\mathbf{x}^0, \mathbf{x}_{k(1)}^1, ..., \mathbf{x}_{k(1)\times ...\times k(t)}^t) \leq 0, & k(1)=1,..., p(1); k(t)=1,..., p(t); \\ 
 &  r=1,...,R_t;  t=2,...,T-1; \\
 & 
\end{array}
\end{equation}
\end{center}
where 
$ F = (\mathbf{f}^0(\mathbf{x}^0); \mathbf{f}_{k(1)}^1(\mathbf{x}^0,\mathbf{x}_{k(1)}^1); ...;\mathbf{f}_{k(1)\times ...\times k(T-1)}^{T-1}(\mathbf{x}^0,\mathbf{x}_{k(1)}^1,...,x_{k(1)\times ...\times k(T-1)}^{T-1}))$, k(t)=1,...,p(t), t=1,..., T-1.

$ \mathbf{x}^0=(x_1^0,...,x_n^0) \in \mathbf{X}^0 $ is an \textit{n}-dimensional initial decision variable vector which made in \texttt{Stage 0} before scenario \textit{k(1)} is happening and $\mathbf{X}^0 $ is an initial decision space. 

$ \mathbf{x}_{k(1)}^1=(x_{1k(1)}^1,...,x_{nk(1)}^1) \in \mathbf{X}^1(\mathbf{x}^0,k(1)),  (k(1)=1,..., p(1)); $ is an (\textit{$n\times p(1)$})-dimensional recourse decision vector to be taken in \texttt{Stage 1} if scenario \textit{k(1)} unfolds and $\mathbf{X}^1(\mathbf{x}^0,k(1)) $ is a recourse decision space when scenario \textit{k(1)} manifests. 

$ \mathbf{x}_{k(1)...k(t)}^t=(x_{1k(1)...k(t)}^t,...,x_{nk(1)...k(t)}^t) \in \mathbf{X}^t(\mathbf{x}^0,k(1),...,k(t))$;
$k(1)=1,..., p(1)$, $k(t)=1, ..., p(t), t=2,...,T-1; $ is an (\textit{$n\times p(1)\times ...\times p(t)$})-dimensional (t=2,...,T-1) recourse decision vector to be taken in \texttt{Stage t} if scenarios \textit{k(1), ..., k(t); (t=2,...,T-1)} have been revealed at \texttt{Stage 1} to \texttt{Stage t} ($t=2,...,T-1$). Therefore, $\mathbf{X}^t(\mathbf{x}^0,k(1),...,k(t)); (t=2,...,T-1) $ is a recourse decision space when scenario \textit{k(1)} to \textit{k(t)} ($t=2,...,T-1$) manifest themselves, respectively. 

 $ $

$ f_i^0(\mathbf{x}^0), (i=1,...,m), $ is $ \mathit{i^{th}}$ objective function in \texttt{Stage 0}.

$ $

$ f_{ik(1)}^1(\mathbf{x}^0,\mathbf{x}_{k(1)}^1), (k(1)=1,...,p(1); i=1,..., m), $ is $ \mathit{i^{th}}$ objective function in \texttt{Stage 1} if scenario \textit{k(1)} happens.

$ $

$ f_{ik(1)...k(t)}^t(\mathbf{x}^0,\mathbf{x}_{k(1)}^1,..., \mathbf{x}_{k(1)...k(t)}^t), (k(1)=1,...,p(1);k(t)=1,...,p(t); t=2,...,T-1; i=1,..., m), $ is $ \mathit{i^{th}}$ objective function in \texttt{Stage t} if scenarios \textit{k(1)} to \textit{k(t)} reveals until this stage.

$ $

$ u_r^0(\mathbf{x}^0) $ is the set of inequality constraints in \texttt{Stage 0}.

$ $

$ u_r^1(\mathbf{x}^0, \mathbf{x}_{k(1)}^1) $ is the set of inequality constraints in \texttt{Stage 1}.

$ $

$ u_r^2(\mathbf{x}^0, \mathbf{x}_{k(1)}^1, ..., \mathbf{x}_{k(1)...k(t)}^t); (t=2,...,T-1); $ is the set of inequality constraints in \texttt{Stage t}.


The problem consists in optimising $[p(1)+(p(1)\times p(2))+...+(p(1)\times ...\times p(t))+1]\times m$ objectives under $ (R_0+(R_1\times p(1))+(R_2\times p(1)\times p(2))+...+ (R_{T-1}\times p(1)\times p(2) \times ...\times p(T-1)) $ constraints. \\


\subsubsection{Multi-stage multi-scenario multi-objective linear programming model}

Model \ref{eq:ms2mop} is a general formulation in which objective functions and constraints can be linear or non-linear, including any type of variables (continuous, integer, or mixed). In the case of linear objective functions and constraints, assuming all the objective functions are to be minimised (maximisation of \textbf{f(x)} is equivalent to minimising $\mathbf{-f(x)} $), the multi-stage multi-scenario multi-objective linear programming (MS2MOLP) model can be formulated as follow:

\begin{center}
\begin{equation}\label{eq:ms2molp}
 \begin{array}{l l}
Min \quad Z_{ik(1)...k(T-1)}(\mathbf{X}) & i=1,...,m; \\
 & k(t)=1,..., p(t); t=1,...,T-1;\\
 
 s.t. \quad  \sum_{j=1}^n a_{rj}^0 x_j^0 \leq b_r^0, & r=1,...,R_0; \\
 & \\

 \sum_{j=1}^n a_{rjk(1)}^0 x_j^0 + \sum_{j=1}^n a_{rjk(1)}^1 x_{jk(1)}^1 \leq b_{rk(1)}^1, & k(1)=1,...,p(1); \\ 
& r=1+...+R_1;  \\
 & \\

  \sum_{j=1}^n a_{rjk(1)..k(t)}^0 x_j^0 + \sum_{j=1}^n a_{rjk(1)..k(t)}^1 x_{jk(1)}^1 + ... + & k(1)=1,...,p(1); \\
& k(t)=1, ...,p(t); \\
 \sum_{j=1}^n a_{rjk(1)..k(t)}^t x_{jk(1)..k(t)}^t \leq b_{rk(1)..k(t)}^t, &  r=1,...,R_t; \\ 
&  t=2,...,T-1; \\
 & \\

& j=1,...,n; \\
 x_j^0, x_{jk(1)}^1, ..., x_{jk(1)..k(t)}^t \geq 0. &  k(1)=1,...,p(1); \\
& k(t)=1, ...,p(t);  \\
 & t=2, ..., T-1. \\
 &
\end{array}
\end{equation}
\end{center}
where, 
\begin{center}
\begin{equation}
 \begin{array}{l l}
Z_{ik(1)..k(T-1)}(\mathbf{X})= \sum_{j=1}^n c_{ij}^0 x_j^0 + (\sum_{j=1}^n c_{ijk(1)}^0 x_j^0 + \sum_{j=1}^n c_{ijk(1)}^1 x_{jk(1)}^1) + & \\
& \\
(\sum_{j=1}^n c_{ijk(1)k(2)}^0 x_j^0 + 
\sum_{j=1}^n c_{ijk(1)k(2)}^1 x_{jk(1)}^1 + \sum_{j=1}^n c_{ijk(1)k(2)}^2 x_{jk(1)k(2)}^2) + ... +  & \\
& \\
(\sum_{j=1}^n c_{ijk(1)k(2)..k(T-1)}^0 x_j^0 + \sum_{j=1}^n c_{ijk(1)k(2)..k(T-1)}^1 x_{jk(1)}^1 + ... + & \\
& \\
\sum_{j=1}^n c_{ijk(1)k(2)..k(T-1)}^{T-1} x_{jk(1)k(2)..k(T-1)}^{T-1}),
& \\
\end{array}
\end{equation}
\end{center}
and demonstrates the $ \mathit{i^{th}}$ linear objective function indicating under conditions pertaining to scenario path $k(1)k(2)..k(T-1)$.

$c_{ij}^0, c_{ijk(1)..k(\theta)}^t, i=1,...,m; j=1,...,n; t=0,..., T-1; \theta=1,...,T-1;$ are given coefficients evaluating the value of $i^{th}$ objective function in relevant scenario path. 

$a_{rj}^0,a_{rjk(1)k(2)..k(t)}^1, a_{rjk(1)k(2)..k(t)}^t, b_r^0,b_{rk(1)}^1, b_{rk(1)k(2)..k(T-1)}^t, r=1,..., R, t=2,..., T-1, j=1,...,n,$ are given coefficient in the constraints.

\subsubsection{Solving Multi-stage multi-objective programming by the reference point goal programming}\label{solve_mop}

As mentioned in Section \ref{back}, in this paper, we use the reference point goal programming (RGP) method \citep{Ogryczak1994} to solve the proposed multi-objective optimization model. If we denote $m\times p(1) \times ... \times p(T-1) $ goals (aspiration levels), to be provided by the decision-maker, by $g_{ik(1)... \times k(T-1)}$ for all meta-objectives, and the corresponding deviation variables by $\delta_{ik(1)..k(T-1)} $, associated constraints can be formulated as follows:

\begin{center}
\begin{equation}
 \begin{array}{l l}
z_{ik(1)..k(T-1)}-\delta_{ik(1)..k(T-1)} =\sum_{j=1}^n c_{ij}^0 x_j^0 + (\sum_{j=1}^n c_{ijk(1)}^0 x_j^0 + \sum_{j=1}^n c_{ijk(1)}^1 x_{jk(1)}^1) + ... + & \\
& \\
(\sum_{j=1}^n c_{ijk(1)..k(T-1)}^0 x_j^0 + \sum_{j=1}^n c_{ijk(1)..k(T-1)}^1 x_{jk(1)}^1 + ... + \sum_{j=1}^n c_{ijk(1)..k(T-1)}^{T-1} x_{jk(1)..k(T-1)}^{T-1}) - & \\ 
& \\
 \delta_{ik(1)..k(T-1)} = g_{ik(1)..k(T-1)}, \quad i=1,..,m, k(1)=1,...,p(1); k(t)=1,...,p(t); t=2, ..., T-1. & \\
& \\
\end{array}
\end{equation}
\end{center}

And the equivalent reference point goal programming model can be formulated as follows:

\begin{center}
\begin{equation}\label{eq:lrpg}
 \begin{array}{l l}
Min \quad \psi = Max_{i, k(1), ..., k(T-1)} \{ \omega_{ik(1)..k(T-1)} \delta_{ik(1)..k(T-1)} \} + & \\
& \\
\quad \quad \varepsilon \sum_{k(T-1)=1}^{p(T-1)} ... \sum_{k(1)=1}^{p(1)} \sum_{i=1}^m (\omega_{ik(1)..k(T-1)} \delta_{ik(1)..k(T-1)}) \\
 
 & i=1,..,m,  \\
 s.t.    \quad z_{ik(1)..k(T-1)} - \delta_{ik(1)..k(T-1)} = g_{ik(1)..k(T-1)}, &  k(t)=1,...,p(t);  \\
 & t=1, ...,T-1; \\
&  \\

 \sum_{j=1}^n a_{rj}^0 x_j^0 \leq b_r^0, & r=1,...,R_0; \\
 & \\
 
  \sum_{j=1}^n a_{rjk(1)}^0 x_j^0 + \sum_{j=1}^n a_{rjk(1)}^1 x_{jk(1)}^1 \leq b_{rk(1)}^1, & k(1)=1,...,p(1); \\ 
& r=1+...+R_1;  \\
 & \\

  \sum_{j=1}^n a_{rjk(1)..k(t)}^0 x_j^0 + \sum_{j=1}^n a_{rjk(1)..k(t)}^1 x_{jk(1)}^1 + & k(1)=1,...,p(1); \\
& k(t)=1, ..., p(t); \\
 ... + \sum_{j=1}^n a_{rjk(1)..k(t)}^t x_{jk(1)..k(t)}^t \leq b_{rk(1)..k(t)}^t, & r=1+...+R_t;  \\ 
&  t=2, ..., T-1;  \\

 &  \\
    x_j^0, x_{jk(1)}^1, x_{jk(1)..k(t)}^t  \geq 0, & j=1, ..., n; k(1)=1, ..., p(1); \\
    & k(t)=1, ..., p(t); t=2, ..., T-1; \\
 &  \\
 
  \delta_{ik(1)..k(t)} \quad \textit{free of sign}, & \forall i, k(t), t=1, ..., T-1, \\
 &
 \end{array}
\end{equation}
\end{center}
where $ \omega_{ik(1)..k(T-1)}\geq 0, (i=1, .., m; k(1)=1, ..., p(1); k(t)=1, ..., p(t); t=2, ..., T-1;);$ are the preferred importance weights of deviations setting by the decision-maker. $\varepsilon $ is an arbitrarily small positive number for the augmented term. 

And by setting $\phi = Max_{i, k(1), ..., k(T-1)} \{ \omega_{ik(1)..k(T-1)} \delta_{ik(1)..k(T-1)} \}$, the linear form can be reformulated as follows:

\begin{center}
\begin{equation}
 \begin{array}{l l}
Min \quad  \psi = \phi  + \varepsilon \sum_{k(T-1)=1}^{p(T-1)} ... \sum_{k(1)=1}^{p(1)} \sum_{i=1}^m (\omega_{ik(1)..k(T-1)} \delta_{ik(1)..k(T-1)})  \\
  & \\
 s.t.    \omega_{ik(1)..k(T-1)} \delta_{ik(1)..k(T-1)} - \phi \leq  0, &      i=1,..,m, k(t)=1,...,p(t);   \\
 & t=1,...,T-1; \\
 & \\
 z_{ik(1)..k(T-1)} - \delta_{ik(1)..k(T-1)} = g_{ik(1)..k(T-1)}, &      i=1,..,m, k(t)=1,...,p(t);   \\
 & t=1,...,T-1; \\
& \\
  \sum_{j=1}^n a_{rj}^0 x_j^0 \leq b_r^0, & r=1,...,R_0; \\
 & \\
   \sum_{j=1}^n a_{rjk(1)}^0 x_j^0 + \sum_{j=1}^n a_{rjk(1)}^1 x_{jk(1)}^1 \leq b_{rk(1)}^1, & k(1)=1,...,p(1); \\ 
& r=1+...+R_1;  \\
 & \\
 
 \sum_{j=1}^n a_{rjk(1)..k(t)}^0 x_j^0 + \sum_{j=1}^n a_{rjk(1)..k(t)}^1 x_{jk(1)}^1 + & k(1)=1,...,p(1); k(t)=1, ..., p(t); \\
& \\
 ... + \sum_{j=1}^n a_{rjk(1)..k(t)}^t x_{jk(1)..k(t)}^t \leq b_{rk(1)..k(t)}^t, & r=1+...+R_t; t=2, ..., T-1; \\ 
 & \\

   x_j^0, x_{jk(1)}^1, x_{jk(1)..k(t)}^t  \geq 0, & j=1, ..., n; k(1)=1, ..., p(1); \\
& k(t)=1, ..., p(t); t=2, ..., T-1; \\
 & \\
 
  \phi, \delta_{ik(1)..k(t)}  \quad \textit{free of sign,} & \forall i, k(t), t=1, ..., T-1;. \\
 &
 
 \end{array}
\end{equation}
\end{center}

The next section compares the two-stage and T-stage (when $T\geq3$) structures to highlight their similarity and differences and find out whether there is any merit in going beyond the two stages.

\subsection{Two-stage \textit{vs.} T-stage structure}\label{2sorTs}

Based on the definitions, the T-stage model includes the two-stage model and some more stages. Thus, the initial and the first contingency decisions of any feasible solution to the T-stage model must also be feasible for the two-stage model. 

\begin{theorem}
Any feasible solution to the T-stage model (\ref{eq:ms2mop}) corresponds to a feasible solution to the two-stage model.
\end{theorem}

\begin{proof}
A mathematical proof can be simply concluded as the T-stage model includes all the constraints of the two-stage model together with some other restrictions. So, the feasible region of the two-stage model involves the image of the feasible region of the three-stage model in this dimension. Therefore, the relevant part of any feasible solution, and also the related part of the Pareto optimal solution, to the T-stage model can satisfy all constraints of the two-stage model and then corresponds to a feasible solution in the two-stage model (Suppose that $\mathbf{X}=(\mathbf{x}^0, \mathbf{x}^1, .., \mathbf{x}^{T-1})$ is a feasible solution for the T-stage model, then $\mathbf{X}^{\prime}=(\mathbf{x}^0, \mathbf{x}^1)$ will be feasible in the two-stage model). 
\end{proof}

\begin{corollary}
The converse is not always true; that is, the feasibility of a solution in the two-stage model does not guarantee the feasibility of the corresponding solution in the T-stage model.
\end{corollary} 

\begin{proof}
Obviously,  the two-stage solution may not satisfy the additional constraints of the T-stage model.
\end{proof}

The above theorem and corollary indicate the fact that although the solution to the two-stage model provides Pareto optimality under the uncertain conditions of the first period, it may not even be feasible and recoverable under additional uncertainties and circumstances of the later stages that would be appended in the T-stage structure ($T\geq3$). Moreover, with a more futuristic vision, the T-stage model may suggest a suboptimal initial decision that could make feasible robust meta-decisions for every plausible scenario path. This property automatically expands the concept of the dynamic-robust solutions for the T-stage framework that naturally existed in the proposed two-stage structure. Indeed, the T-stage model ($T\geq3$) provides a more robust initial solution than the two-stage model. Of course, it may not be Pareto optimal if only the first stage is considered, but the dynamic-robust meta-decisions will be both feasible and Pareto optimal across the T-stage model ($T\geq3$).

The two-stage structure always plans and looks one step ahead. In contrast, the T-stage structure looks further ahead and considers the conditions and consequences of T-1 steps in advance. Consequently, the T-stage model can consider longer-term consequences and find a more robust solution, if available. However, this robustness can be reached at the price of additional complexity and computation, making the efficiency of the T-stage model unclear. Note that, here, by the efficiency of the T-stage, we mean the added value of finding a more robust solution is more than the expenses of the extra complexity and computation. Since the number of scenario paths increases exponentially, considering many stages could result in a model that is too complex, computationally expensive and challenging to solve. 

To answer the above question, we compare the solutions generated by two- and T-stage models in this study. For a fair comparison, a similar time window must be considered, which is a fundamental difference between the two models. 
Thus, on the one hand, we compare the robustness of the initial solutions generated by the two models. On the other hand, we iterate the two-stage model, T-1 times, from the second stage to cover the whole time horizon of the T-stage model and be able to compare the solutions generated by the two models. Indeed, for the latter comparison, in the next section, we will introduce a \emph{``moving horizon structure''} where, first, the initial decision of the two-stage model will be implemented; then, after scenario realisation, another two-stage model will be run, which can cover the decisions of the second and third stages. The process will be continued until the entire T stages are covered.  

 In the rest of the paper, we focused on linear programming and a three-stage structure (T=3) to avoid too many complexities and gain better insights into the context without losing generality. All the concepts, models and theorems can be simply extended to more than three stages.

\subsection{Two-stage moving horizon approach}\label{2smoving}

Consider a decision-making problem with a three-stage planning horizon that includes three stages of decision-making and two steps of scenario realisation, as shown at the top of Figure \ref{fig:3sand22s}. Suppose that $\mathbf{x}^0$, $\mathbf{x}^1$, and $\mathbf{x}^2$ present the decision vectors related to each stage, respectively. $S(k_1)$ and $S(k_2)$ indicate the scenario spaces regarding the first and the second steps of scenario revelation, respectively. In a \textit{moving horizon} model, we always plan as if it is a two-stage horizon. This means that, at first, a two-stage model is used to generate the initial decisions ($\mathbf{x}^0$) that could be immediately implemented. Then, we need to wait and see which scenario from $S(k_1)$ will unfold. After that, another two-stage model starting at the second stage is applied to reach the contingent decisions ($\mathbf{x}^1$ and $\mathbf{x}^2$). That is, the second two-stage ($2 \times two-stage$). In other words, we roll the two-stage structure continuously.

Note that the contingency decisions ($\mathbf{x}^1$) generated by the first two-stage model are not implemented, and they may be used as an approximation that will substitute with the initial decisions gained from the new two-stage model applying after scenario realisation in the first stage. In other words, after the iteration, the contingency decisions (generated by the first two-stage model) change their role and become the initial decisions of the second two-stage model. This can help us to adapt the decision in the next stage if necessary. 

 Moreover, to be able to generate all the contingency decisions, for comparison purposes, one needs to run the second two-stage model for every plausible scenario in the $S(k_1)$ set (one for each), although in practice the second two-stage model will only run once for the actually realised scenario. Figure \ref{fig:3sand22s} demonstrates the structure of the two-stage moving horizon models and their generated solutions in comparison with the three-stage model. 
  
 $ $ \\

\begin{figure}[!htbp]

\tikzstyle{decision} = [diamond, draw, fill=blue!20, 
    text width=5.5em, text badly centered, node distance=3cm, inner sep=0pt]
\tikzstyle{block} = [rectangle, draw, fill=blue!20, 
    text width=3em, text centered, rounded corners, minimum height=2em]
\tikzstyle{line} = [draw, -latex']
\tikzstyle{cloud} = [draw, circle,fill=red!20, node distance=1.5cm,
    minimum height=1.5em]
\tikzstyle{emt} =[node distance=1cm]

\begin{center}
   \small 
\begin{tikzpicture}[node distance = 1cm, auto]
   
    \draw[gray, thick, dashed] (-1,3) -- (8,3);
   
    \node [block] (x0) at (0,2) {$X^0$};
    \draw[red,thick, dashed] (0,2) circle (10pt);
    \node [cloud] (sk1) at (1.5,2) {$S_{k(1)}$};    
    \node [block] (x1) at (3,2) {$X^1$};
    \draw[red,thick, dashed] (3,2) circle (10pt);
    
    \node [cloud] (sk2) at (4.5,2) {$S_{k(2)}$};
     \node [block] (x2) at (6,2) {$X^2$};
     \draw[red,thick, dashed] (6,2) circle (10pt);
    
    \path [line] (x0) --   (sk1);   
    \path [line] (sk1) --  (x1);
    
     \path [line] (x1) --   (sk2);   
    \path [line] (sk2) --  (x2);
    
    \draw[gray, thick, dashed] (-1,1) -- (8,1);

    \node [block] (x0) at (0,0) {$X^0$}; 
       \draw[red,thick, dashed] (0,0) circle (10pt);
    \node [cloud] (sk1) at (1.5,0) {$S_{k(1)}$};    
    \node [block] (x1) at (3,0) {$X^1$};
    
     \node [block] (x1n) at (3,-1) {$X^1$};
     \draw[red,thick, dashed] (3,-1) circle (10pt);
     \node [cloud] (sk2) at (4.5,-1) {$S_{k(2)}$};
     \node [block] (x2) at (6,-1) {$X^2$};
     \draw[red,thick, dashed] (6,-1) circle (10pt);
     
     \path [line] (x0) --   (sk1);   
     \path [line] (sk1) --  (x1);
    
     \path [line] (x1n) --   (sk2);   
    \path [line] (sk2) --  (x2);
    
    \draw[gray, thick, dashed] (-1,-2) -- (8,-2);
    
\end{tikzpicture}
\caption{Two-stage moving horizon structure compared to the three-stage structure}\label{fig:3sand22s} 
\end{center}
\end{figure}
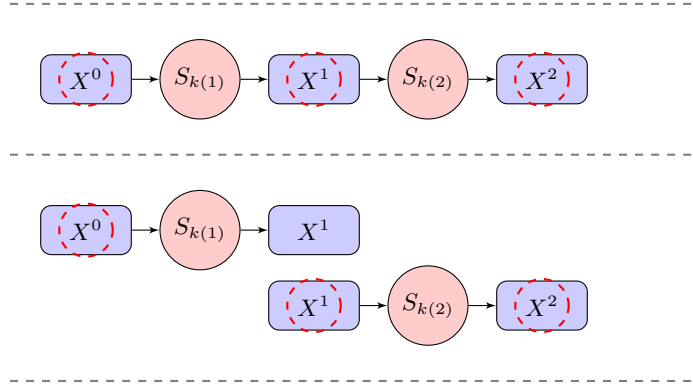

Therefore, if scenario $k(1)$ manifested itself at the first stage, the second two-stage model regarding this scenario can be formulated as follows:

\begin{center}
\begin{equation}\label{eq:22s}
 \begin{array}{l l}
Min \quad Z_{ik(1)}=  \sum_{j=1}^n c_{ijk(1)}^1 x_{jk(1)}^1 + \sum_{j=1}^n c_{ijk(1)k(2)}^2 x_{jk(1)k(2)}^2;  & i=1,...,m; k(2)=1,...,p(2);  \\
 & \\
 s.t.   \sum_{j=1}^n a_{rjk(1)}^1 x_{jk(1)}^1 \leq \beta_{rk(1)}^1, & r=1,...,R_{k(1)}; \\
 & \\
  \sum_{j=1}^n a_{rjk(1)k(2)}^1 x_{jk(1)}^1 + \sum_{j=1}^n a_{rjk(1)k(2)}^2 x_{jk(1)k(2)}^2\leq \beta_{rk(1)k(2)}^2, & k(2)=1, ...,p(2); \\ 
& r=R_{k(1)}+1,...,R_{k(1)}+...+R_{k(1).k(2)};   \\
 & \\
    x_{jk(1)}^1, x_{jk(1)k(2)}^2 \geq 0. & j=1,...,n; k(1)=1, ...,p(1), \\
 & k(2)=1, ...,p(2).\\
 &
\end{array}
\end{equation}
\end{center}
$ $\\
where \small $\beta_{rk(1)}^1=b_{rk(1)}^1 - \sum_{j=1}^n a_{rjk(1)}^0 x_j^{0*}$, $\beta_{rk(1)k(2)}^2=b_{rk(1)k(2)}^2 - \sum_{j=1}^n a_{rjk(1)k(2)}^0 x_j^{0*}$ and $x_j^{0*}, (j=1,..., n);$ is the most preferred Pareto optimal solution generated by the first two-stage model.

\begin{defn}{\textbf{(Feasible solutions in a two-stage moving horizon model).}}\label{def3}
\begin{small}
\textit{Vector} $\mathit{\mathbf{X}=(\underline{\mathbf{x}}^0, \overline{\mathbf{x}}_{k(1)}^1, \overline{\mathbf{x}}_{k(1)k(2)}^2)}$ \textit{is feasible for a two-stage moving horizon model \textbf{if} it simultaneously satisfies the constrains in both two-stage models
, and also} $\mathit{\forall k=k(1) \in S^1}$, \textit{and} $\mathit{\forall k(2)\in S^2, \exists \underline{\mathbf{X}}=(\underline{\mathbf{x}}^0, \underline{\mathbf{x}}_{k(1)}^1)}$\textit{, and} $\mathit{ \overline{\mathbf{X}}=(\overline{\mathbf{x}}_{k(1)}^1, \overline{\mathbf{x}}_{k(1)k(2)}^2), }$ \textit{s.t.} $\mathit{\underline{\mathbf{X}}}$ \textit{and} $\mathit{\overline{\mathbf{X}}}$ \textit{are feasible for the first and the second 
two-stage models, respectively.}
\end{small}
\end{defn}

\begin{theorem}\label{22sin3s}
Any feasible solution 
to the two-stage moving horizon model 
in a three-stage planning window is a feasible solution for the corresponding three-stage model 
and \textit{vice versa}---i.e., the two-stage moving horizon model is feasible \textbf{\emph{iff}} the three-stage model is feasible.
\end{theorem}

\begin{proof}
The proof of this theorem can be found in Appendix \ref{proof22sin3s}.
\end{proof}

Theorem \ref{22sin3s} shows that any feasible solution to the two-stage moving horizon model, based on definition \ref{def3}, is feasible for the three-stage model and \emph{vice versa}. However, in the first two-stage model, we cannot consider the initial decisions' consequences and the scenarios' conditions in the third stage. Therefore, Pareto optimal initial decisions ($\mathbf{x}^0$) generated by the first two-stage model may give rise to a situation in which, in some scenario paths when running the second two-stage model, we cannot find any solutions to build a Pareto optimal or feasible meta-solution to the two-stage moving horizon model in combination with that initial decision generated by the first two-stage model. Therefore, some solutions generated by the two-stage moving horizon models might be infeasible or dominated by a solution generated by the three-stage model.

However, this issue will not occur if the three-stage model is applied, highlighting the robustness of the decisions in the proposed three-stage structure and confirming that the initial solutions provided by the proposed three-stage structure are more robust than the initial solutions of the two-stage model (for further illustration see example \ref{exmp1} describing the issue mentioned above in Appendix \ref{app2}).
However, extra computation and cognitive costs needed to gain these superiorities should be justified, particularly in complex, large-scale real-life problems. 

In the next section, we examine and compare the applicability of both the proposed T-stage (T=3) and moving horizon structures in a sequential portfolio problem. 

\section{An application to sequential portfolio selection under deep uncertainty.}\label{portfolio}

 Consider an investment portfolio with five stocks (or investment options) $I_1,...,I_5$, 
 which differ in terms of both risk and growth of funds. Furthermore, suppose that an amount of €$5$ $000$ $000$ is available and a minimum of €$250$ $000$ must be withdrawn at each stage under any circumstances. The maximum withdrawal is €$1$ $500$ $000$.

Assume that the economists identified five states $S_1,...,S_5$ (high inflation, stagflation, baseline inflation, deflation, and low inflation) in any one stage representing the set of plausible scenarios\footnote{Note that to simulate deep uncertainty, we shall not specify scenario's probabilities.}  which depend on the scenario in the previous stage, with the transition possibilities shown in Table~\ref{tab:trans}.

\begin{table}[h!]
\begin{center}
\begin{tabular}{cccccc} \toprule
 Previous stage & \multicolumn{5}{c}{State at next stage}\\ \cline{2-6}
 State & $S_1$ & $S_2$ & $S_3$ & $S_4$ & $S_5$ \\ \hline
 $S_1$ & \checkmark & \checkmark & & & \\
 $S_2$ & \checkmark & \checkmark & \checkmark & & \\
 $S_3$ & & \checkmark & \checkmark & \checkmark & \\
 $S_4$ & & & \checkmark & \checkmark & \checkmark \\
 $S_5$ & & & & \checkmark & \checkmark \\ \bottomrule 
\end{tabular}
\end{center}
\caption{Possible transition between different scenarios}\label{tab:trans}
\end{table}

 The anticipated growth in funds over each state has been portrayed in Table \ref{tab:grow}. For each stock, there is an opportunity to withdraw part of the investment as cash in each stage, which could either be spent for current costs or be transferred to another stock for re-investment. However, some percentages of the total fund will be lost upon transfer between stocks, and this loss could be state-related, as shown in Tables \ref{tab:penS1}-\ref{tab:penS5}. Otherwise, the money remains until the next investment maturity at the next stage. No switching of funds between stocks is ever possible between stages.  

\begin{table}[!htbp]
\begin{center}
\begin{tabular}{cccccc} \toprule
  & \multicolumn{5}{c}{Growth under scenario}\\ \cline{2-6}
 Investment & $S_1$ & $S_2$     & $S_3$       & $S_4$      & $S_5$ \\ \hline
 $I_1$ & $ -20\% $ & $ +4\% $   & $ +16\% $   & $ +20\% $  & $ +50\% $ \\ \hline
 $I_2$ & $ -2\% $  & $  +8\% $  & $ +11.5\% $ & $ +20\% $  & $ +30\% $ \\ \hline
 $I_3$ & $ +8\% $  & $ +8.5\% $ & $ +9\% $    & $ +9.5\% $ & $ +10\% $ \\ \hline
 $I_4$ & $  +4\% $ & $ +7\% $   & $ +12\% $   & $ +16\% $  & $ +20\% $ \\ \hline
 $I_5$ & $ -15\% $ & $ +6\% $   & $ +15\% $   & $ +20\% $  & $ +35\% $  \\ \bottomrule
\end{tabular}
\end{center}
\caption{Percentage growths for each stock under each scenario}\label{tab:grow}
\end{table}

\begin{table}[!h]
\thisfloatpagestyle{empty}
\begin{center}
\begin{tabular}{|cccccc|c|} \toprule
 Transfer & \multicolumn{5}{c|}{to} & \\ \cline{2-6}
 from  & $I_1$       & $I_2$       & $I_3$      & $I_4$       & $I_5$       & Withdrawal \\ \hline
 $I_1$ & $0$         & $ -2.5\% $  & $ -3\% $   & $ -3\% $    & $ -2\%$     & $ -3\% $ \\ \hline
 $I_2$ & $ -0.05\%$  & $0$         & $ -1\% $   & $ -0.1\% $  & $ -0.1\% $  & $ -0.3\%$ \\ \hline
 $I_3$ & $ -0.01\%$  & $ -0.1\% $  & $0$        & $ -0.01\% $ & $ -0.01\% $ & $ -0.1\%$ \\ \hline
 $I_4$ & $ -0.01\%$  & $ -0.01\% $ & $ -0.8\% $ & $0$         & $ -0.01\% $ & $ -0.2\%$ \\ \hline
 $I_5$ & $ -0.1\%$   & $ -2.5\% $  & $ -3\% $   & $ -3\% $    & $0$         & $ -2.5\%$\\ \bottomrule
\end{tabular}
\caption{Percentage of loss of funds (penalty cost) for transferring between each pair of stock and withdrawal in state $S_1$}\label{tab:penS1}

\begin{tabular}{|cccccc|c|} \toprule
 Transfer & \multicolumn{5}{c|}{to} & \\ \cline{2-6}
 from & $I_1$      & $I_2$      & $I_3$      & $I_4$       & $I_5$      & Withdrawal \\ \hline
 $I_1$  & $0$      & $ -1\% $   & $ -1.2\% $ & $ -1.0\% $  & $ -0.7\%$  & $ -2\% $ \\ \hline
 $I_2$ & $ -0.5\%$ & $0$        & $ -1.0\% $ & $ -0.5\% $  & $ -0.3\% $ & $ -0.4\%$ \\ \hline
 $I_3$ & $ -0.7\%$ & $ -0.2\% $ & $0$        & $ -0.01\% $ & $ -0.2\% $ & $ -0.3\%$ \\ \hline
 $I_4$ & $ -0.5\%$ & $ -1\% $   & $ -1.5\% $ & $0$         & $ -0.1\% $ & $ -0.4\%$ \\ \hline
 $I_5$ & $ -0.2\%$ & $ -1\% $   & $ -1.5\% $ & $ -0.1\% $  & $0$        & $ -1.5\%$\\ \bottomrule
\end{tabular}
\caption{Percentage of loss of funds (penalty cost) for transferring between each pair of stock and withdrawal in state $S_2$}\label{tab:penS2}

\begin{tabular}{|cccccc|c|} \toprule
 Transfer & \multicolumn{5}{c|}{to} & \\ \cline{2-6}
 from  & $I_1$     & $I_2$      & $I_3$      & $I_4$       & $I_5$      & Withdrawal \\ \hline
 $I_1$ & $0$       & $ -0.4\% $ & $ -0.5\% $ & $ -0.3\% $  & $ -1.0\%$  & $ -1.0\% $ \\ \hline
 $I_2$ & $ -1.1\%$ & $0$        & $ -0.2\% $ & $ -0.01\% $ & $ -1.1\% $ & $ -1.2\%$ \\ \hline
 $I_3$ & $ -1.2\%$ & $ -1\% $   & $0$        & $ -0.3\% $  & $ -1.0\% $ & $ -2\%$ \\ \hline
 $I_4$ & $ -1.1\%$ & $ -1.5\% $ & $ -0.7\% $ & $0$         & $ -1.0\% $ & $ -2\%$ \\ \hline
 $I_5$ & $ -0.8\%$ & $ -0.3\% $ & $ -0.3\% $ & $ -0.2\% $  & $0$        & $ -0.8\%$\\ \bottomrule
\end{tabular}
\caption{Percentage of loss of funds (penalty cost) for transferring between each pair of stock and withdrawal in state $S_3$}\label{tab:penS3}

\begin{tabular}{|cccccc|c|} \toprule
 Transfer & \multicolumn{5}{c|}{to} & \\ \cline{2-6}
 from  & $I_1$    & $I_2$       & $I_3$       & $I_4$       & $I_5$     & Withdrawal \\ \hline
 $I_1$ & $0$      & $ -0.01\% $ & $ -0.01\% $ & $ -0.01\% $ & $ -0.5\%$ & $ -0.1\% $ \\ \hline
 $I_2$ & $ -2\%$  & $0$         & $ -0.1\% $  & $ -0.1\% $  & $ -2\% $  & $ -1.5\%$ \\ \hline
 $I_3$ & $ -3\%$  & $ -2.5\% $  & $0$         & $ -0.7\% $  & $ -3\% $  & $ -2.5\%$ \\ \hline
 $I_4$ & $ -3\%$  & $ -2\% $    & $ -0.1\% $  & $0$         & $ -3\% $  & $ -2.5\%$ \\ \hline
 $I_5$ & $ -5\%$ & $ -0.01\% $ & $ -0.01\% $ & $ -0.01\% $ & $0$       & $ -0.1\%$\\ \bottomrule
\end{tabular}
\caption{Percentage of loss of funds (penalty cost) for transferring between each pair of stock and withdrawal in state $S_4$}\label{tab:penS4}

\begin{tabular}{|cccccc|c|} \toprule
 Transfer & \multicolumn{5}{c|}{to} & \\ \cline{2-6}
 from  & $I_1$      & $I_2$       & $I_3$       & $I_4$       & $I_5$      & Withdrawal \\ \hline
 $I_1$ & $0$        & $ -0.01\% $ & $ -0.01\% $ & $ -0.01\% $ & $ -1.5\%$  & $ -0.2\% $ \\ \hline
 $I_2$ & $ -1.5\%$  & $0$         & $ -0.05\% $ & $ -0.1\% $  & $ -2.5\% $ & $ -1.5\%$ \\ \hline
 $I_3$ & $ -3\%$    & $ -2.5\% $  & $0$         & $ -1\% $    & $ -3\% $   & $ -2.5\%$ \\ \hline
 $I_4$ & $ -2.5\%$  & $ -2\% $    & $ -0.1\% $  & $0$         & $ -3\% $   & $ -2.5\%$ \\ \hline
 $I_5$ & $ -0.01\%$ & $ -0.01\% $ & $ -0.01\% $ & $ -0.01\% $ & $0$        & $ -0.1\%$\\ \bottomrule
\end{tabular}
\caption{Percentage of loss of funds (penalty cost) for transferring between each pair of stock and withdrawal in state $S_5$}\label{tab:penS5}
\end{center}
\end{table}

Moreover, two objective functions ($ f_n, n=1,2$) have been considered as follows:
\begin{enumerate}
\item[$ \mathbf{f_1}$:] Maximising the desirable level of total funds available after withdrawals.
\item[$ \mathbf{f_2}$:] Maximising the cash withdrawals at each stage (i.e., at the end of the stage) between an absolute minimum and a desirable maximum, with goals that may be state-dependent.
\end{enumerate}

\begin{figure}[!ht]
  
  \centering
  \includegraphics[width=0.65 \textwidth]{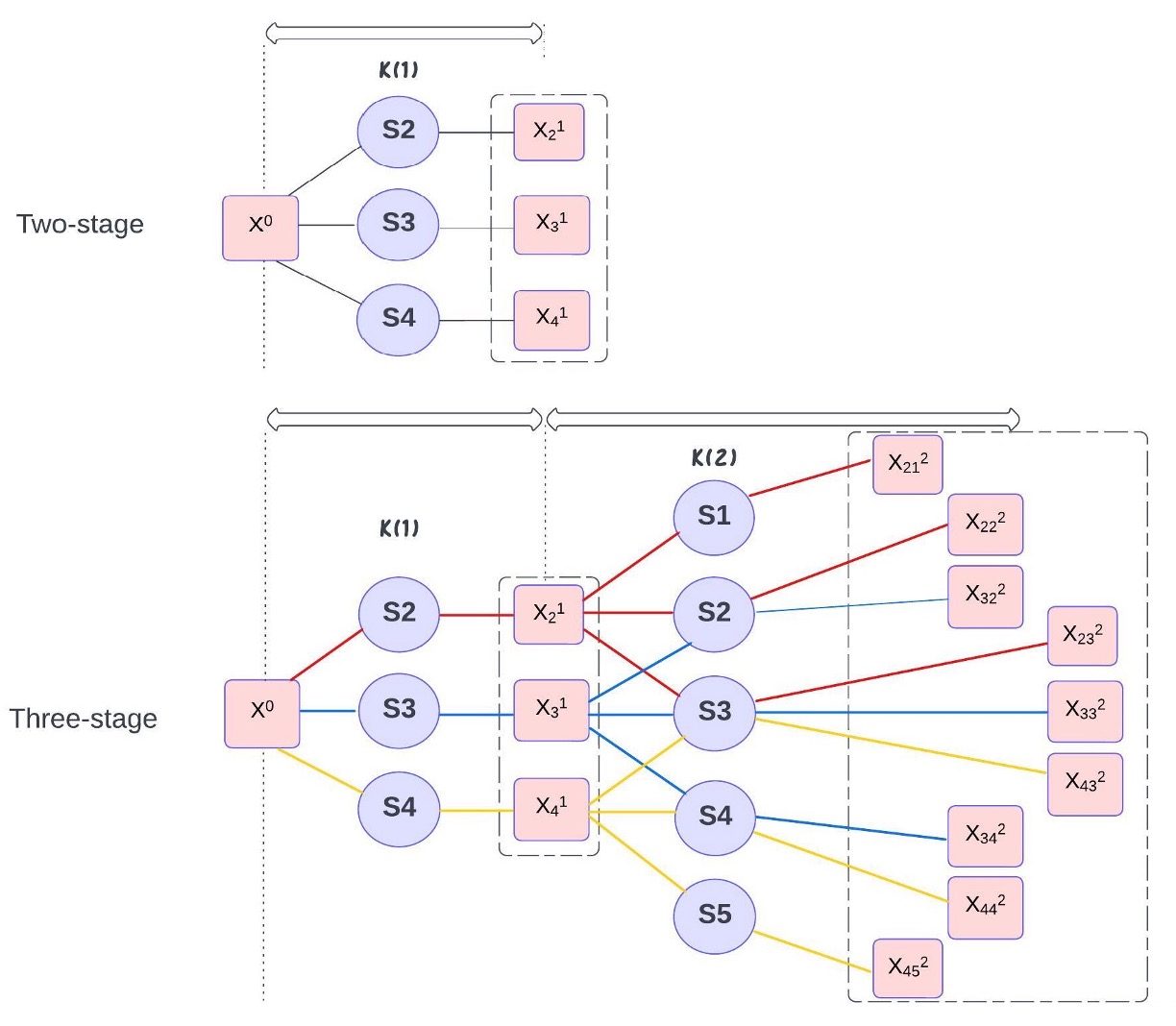}

\caption{Scenarios of the two-stage structure in comparison with meta-scenarios of the three-stage structure}
\label{fig:e3scenarios}
\end{figure}

 Suppose that, currently, we are facing the economic situation of baseline inflation (i.e., state $S_3$). Then, in the two-stage structure, scenarios are collected from $K_1=\{S_2, S_3, S_4\}$. However, in the three-stage structure, scenario paths can be set as follows:  $k(1) \in K_1=\{S_2, S_3, S_4\}$ and $k(2)\in K_2$, but $K_2$ would be dependent on $K_1$. If $K_1=S_2$ then $K_2=\{S_1, S_2, S_3\}$ and for $K_1=S_3$ and $K_1=S_4$, $K_2$ would be respectively equal to $\{S_2, S_3, S_4\}$ and $\{S_3, S_4, S_5\}$. Figure \ref{fig:e3scenarios} describes and compares scenarios of the two-stage structure and scenario paths of the three-stage structure as well as their stages and solutions.

Accordingly, there are three scenarios in a two-stage structure that could establish six meta-objectives in combination with the two objective functions. In a three-stage structure, each of these three scenarios would be followed by a different group of three plausible scenarios that will construct nine scenario paths in which the objective values must be evaluated and compared. In other words, alternative decisions must be evaluated in terms of all eighteen (six, in the two-stage structure) meta-objectives which will give us nine (three) meta-decisions. Although finding the best group of nine (three) meta-decisions that simultaneously satisfy all eighteen (six) meta-objectives is, generally, impossible, 
a Pareto optimal decision would be reached by solving the three-stage (two-stage) model. Therefore, considering the notations in Appendix \ref{tab:var}, the three-stage multi-objective optimisation model for this problem can be formulated as follows :


\begin{center}
\begin{equation}\label{eq:e33s}
 \begin{array}{l l}
Max \quad Z_{1k(1)k(2)}= \sum_{j=1}^5 \sum_{i=1}^5 (x_{ij}^0 + x_{ijk(1)}^1)  + \sum_{i=1}^5 x_{i7k(1)k(2)}^2 & \forall k(1), k(2)\\
&\\
Max \quad Z_{2k(1)k(2)}= \sum_{i=1}^5 (x_{i6}^0 +  x_{i6k(1)}^1 +  x_{i6k(1)k(2)}^2) & \forall k(1), k(2)\\
& \\

 s.t. &\\ 
 
& \\
 
(\textit{Funds  balance constraints}) &\\
&\\
   \sum_{j=1}^6 (1+p_{ij})x_{ij}^0 = b_i^0, & i=1,...,5; \\
 & \\
  \sum_{j=1}^6 (1+p_{ij}) x_{ijk(1)}^1 = \sum_{j=1}^5 (1+c_{ik(1)}) x_{ji}^0   & \forall i, k(1);\\
 & \\
 
    (1+p_{i6}) x_{i6k(1)k(2)}^2 + x_{i7k(1)k(2)}^2 =\sum_{i=1}^5 (1+c_{jk(1)k(2)}) x_{jik(1)}^1   & \forall i, k(1), k(2); \\ 
 & \\

 (\textit{Minimum  withdrawal constraints})&\\
 &\\
 \sum_{i=1}^5  x_{i6}^0 \geq b_{6}^0, &  \\ 
 &\\
  \sum_{i=1}^5  x_{i6k(1)}^1 \geq b_{6k(1)}^1, & \forall k(1); \\ 
&\\
 \sum_{i=1}^5  x_{i6k(1)k(2)}^0 \geq b_{6k(1)k(2)}^2, & \forall k(1), k(2); \\ 
 & \\

 (\textit{Non-negativity  constraints})&\\
 &\\
   x_{ij}^0, x_{ijk(1)}^1, x_{i6k(1)k(2)}^2, x_{i7k(1)k(2)}^2 \geq 0. & i=1,...,5; j=1,...,6; \\
& \forall k(1), k(2). \\
 & 
\end{array}
\end{equation}
\end{center}


$ $

By eliminating the third stage and its relevant states and variables from the three-stage model, the two-stage multi-objective optimisation model for this problem can be formulated (see model \ref{eq:e32s} in Appendix \ref{2s_portfolio}). Parameters and decision variables that are utilised to formulate the two- and three-stage models are listed in Appendix \ref{tab:var}, as well as the relevant GP model for both (see models \ref{eq:e32sGP} and \ref{eq:e33sGP} in Appendix \ref{3sRGP} and \ref{2sRGP}). 

 Clearly, any variation of states in the third stage will not affect the first two-stage initial decision. For example, the first two-stage model has no idea about states $S_1$ and $S_5$ and the initial decision on this model, in some cases, may not work well enough in the future if one of those states unfolded. Such a variation may influence the initial decision of the three-stage model. As propounded in \ref{2sorTs}, future decisions might be affected by current decisions in a multi-stage problem. Therefore, considering different plausible futures and consequences of current decisions on each future that the three-stage structure could provide is helpful in such a problem---i.e., it generates a more robust initial solution than the two-stage structure.

\subsection{Numerical comparison of solutions generated by the three-stage and two-stage moving horizon models}

Considering the goals (aspiration levels) 
represented in Table \ref{tab:goals}, provided by the decision-maker, we solved the three-stage model (model \ref{eq:e33s}) and relevant models for the corresponding two-stage moving horizon structure via the reference point goal programming (RGP) method \citep{Ogryczak1994}, described in \ref{GP} and \ref{solve_mop} (corresponding three- and two-stage models can be found in Appendix \ref{3sRGP} and \ref{2sRGP}, respectively).

 \begin{table}[!htbp]
\begin{center}
\begin{tabular}{ccccccc} \toprule
 & \multirow{2}{*}{Goals} & \multicolumn{5}{c}{Plausible states (scenarios)}\\ \cline{3-7}
 &        & $S_1$   & $S_2$      & $S_3$    & $S_4$   & $S_5$ \\ \hline
 \multirow{2}{*}{\texttt{stage 0}} & $g_1^0$ & $ -  $  & $  - $  & $ 5.5 $  & $ - $   & $ - $ \\ \cline{2-7}
                            & $g_2^0$ & $ - $ & $ - $    & $ 0.75 $ & $   - $ & $ - $ \\ \hline
                            
 \multirow{2}{*}{\texttt{stage 1}} & $g_1^1$ & $ -  $  & $  6.5 $   & $ 7 $  & $ 7.5 $   & $ - $ \\ \cline{2-7}
                            & $g_2^1$ & $ - $ & $ 0.5 $    & $ 0.75 $ & $   1 $ & $ - $ \\ \hline
                            
 \multirow{2}{*}{\texttt{stage 2}} & $g_1^2$ & $ 7  $  & $  7.5 $   & $ 8 $  & $ 9 $   & $ 11 $ \\ \cline{2-7}
                            & $g_2^2$ & $ 0.5 $ & $ 0.5 $    & $ 0.75 $ & $ 1 $ & $ 1.5 $ \\ \bottomrule
                            
 \end{tabular}
\end{center}
\caption{Desirable levels of total remained funds (million Dollars) after consumption in each state}\label{tab:goals}
\end{table}

In the two-stage moving horizon model (shown in figure \ref{fig:e3scenarios1}), 
the initial decision of the first two-stage model is implemented. Then, after scenario realisation, another two-stage model will be run, which can cover the decisions of the second and third stages. To evaluate and compare the solutions of these two following two-stage models with the three-stage model's solutions, we need to run the second two-stage model three times, one for each plausible scenario, to produce nine meta-decisions for nine scenario paths. 

Indeed, after identifying the initial decision from the first two-stage model, we will suppose that one scenario unfolded in the first scenario realisation section; say, e.g., $S_{k(1)}; k(1)=\{S_2, S_3, S_4\}$. Then, the consequences of the initial decision under conditions of state $S_{k(1)}; k(1)=\{S_2, S_3, S_4\}$ are simulated, and the second two-stage model with an updated amount of fund in each stock will be run. 
The initial decisions of these three new two-stage models will be considered as the first set of scenario-relevant contingent decisions ($X_2^1$, $X_3^1$, and $X_4^1$), while the contingent decisions generated by the second two-stage models will represent the scenario-relevant second contingent decisions ($X_{21}^2, ..., X_{45}^2$).

\begin{figure}[!t]
  
  \centering
  \includegraphics[width=0.65 \textwidth]{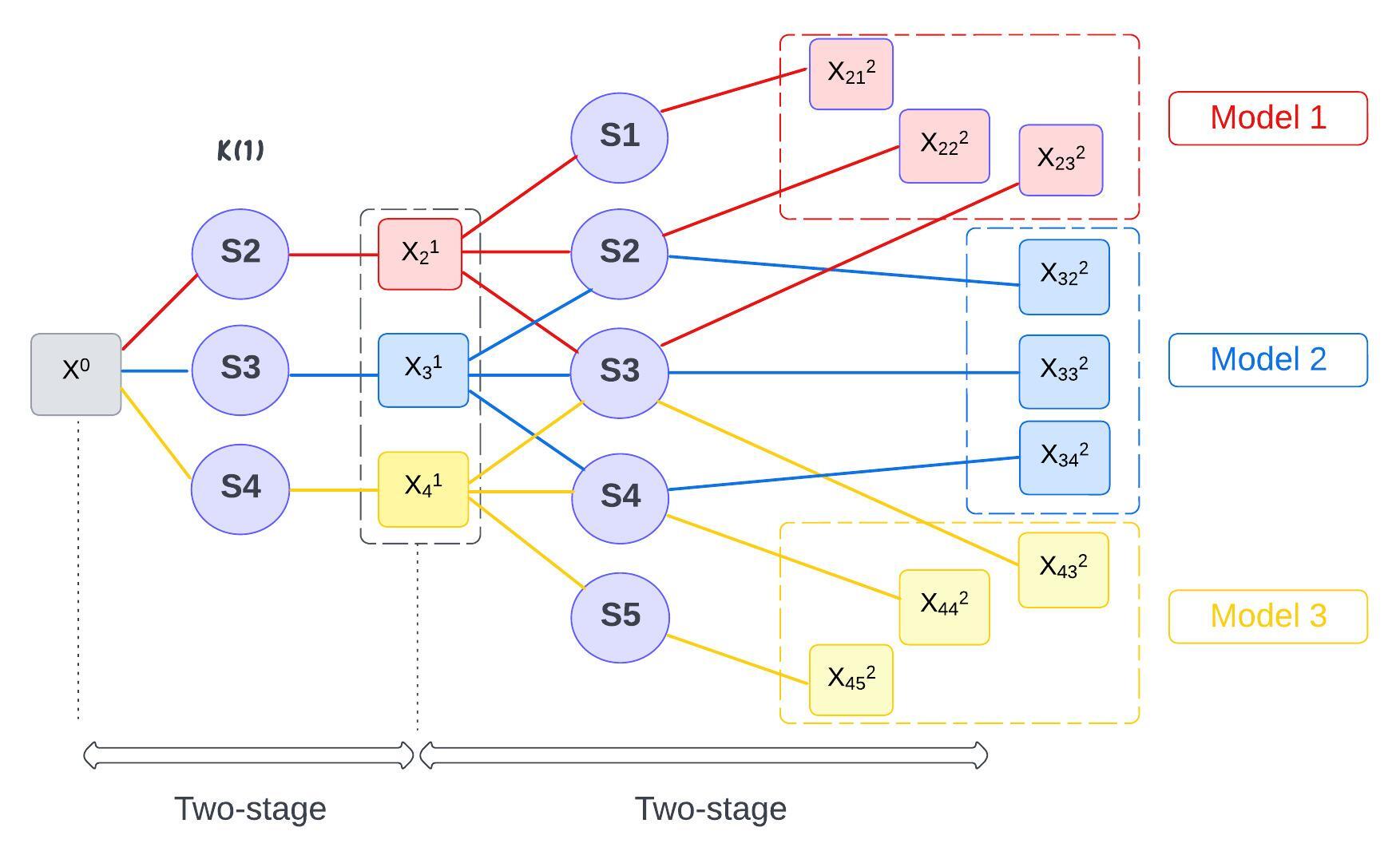}

\caption{Scenarios of the 2 $\times$ two-stage structure}
\label{fig:e3scenarios1}
\end{figure} 

Accordingly, as shown in figure \ref{fig:e3scenarios1}, if $S_2$ is revealed, then solution $X_2^1$ will be implemented, and the set of plausible scenarios in the next period starting at this stage will be $\{S_1, S_2, S_3\}$.  $\{S_2, S_3, S_4\}$ determines the set of plausible scenarios in the next stage if $S_3$ unfolds and $X_3^1$ is executed. Also, in the case of $S_4$ realisation, $X_4^1$ will be performed, and the plausible scenarios can be presented by $\{S_3, S_4, S_5\}$. Thus, by solving these three two-stage models 
 comparable solutions with the three-stage model will be generated. 

Table \ref{tab:3s} compares Pareto optimal solutions generated by relevant three-stage and two-stage moving horizon models utilising the reference-point goal programming approach and the given preferences.
As shown in this table (Table \ref{tab:3s}), the second column indicates the suggested investment amount for each stock at the beginning of the horizon relating to the initial decision ($X^0$) based on the given preferences.  
The first contingency action, which will be implemented at the end of the second stage, has been described in the third column.  
Available capital at the end of the related time horizon in each structure has been portrayed in the fourth column. 

Finally, the last four columns include the other contingency actions (i.e., withdrawal) at each stage, together with the total withdrawal at the end of the investment window. The withdrawal in the first, second, and third stages are demonstrated by $X_6^0$, $X_6^1$, and $X_6^2$, respectively.

\begin{landscape}   
\begin{table}[!htbp]
\begin{center}
\resizebox{0.99\textwidth}{!}{

\begin{tabular}{|c|c|c|c|c|c|c|c|} \hline
\multirow{2}{*}{Structure} & Initial decision & First contingency &  Total worth of investment & \multicolumn{4}{|c|}{Withdrawal} \\
 & $X^0=(I_1, I_2, I_3, I_4, I_5)$ & $X_{k(1)}^1=(I_1, I_2, I_3, I_4, I_5)$ & at the end of the third stage & $X_6^0$ & $X_6^1$ & $X_6^2$ & Total \\ \hline
 
\multirow{11}{*}{3-stage} & \multirow{11}{*}{(1 000 000, 1 000 000, 0, 0, 2 728 200)} & \multirow{3}{*}{(35 800, 120'200, 4 241 100, 277 700, 48 200)} & 3 697 200 & \multirow{9}{*}{250 000} & \multirow{3}{*}{250 020} & 1 357 300 & 1 857 300 \\ 
 & & & 4 127 600 & & & 985 500 & 1 485 600 \\ 
 & & & 4 791 900 & & & 366 600 & 866 700 \\
 & & & & & & & \\
 & & \multirow{3}{*}{(1 205 900, 953 300, 93 300, 242 400, 2 650 000)} & 4 521 000 & & \multirow{3}{*}{250 150} & 924 700 & 1 424 800 \\
 & & & 4 944 700 & & & 927 900 & 1 428 000 \\
 & & & 5 352 800 & & & 797 800 & 1 297 900 \\
  & & & & & & & \\
 & & \multirow{3}{*}{(5 172 200, 0, 0, 0, 0)} & 5 123 500 & & \multirow{3}{*}{500 000} & 867 600 & 1 617 600 \\
 & & & 5 583 700 & & & 622 300 & 1 372 300 \\
 & & & 7 507 800 & & & 250 000 & 1 000 000 \\ \hline \hline
 
\multirow{11}{*}{2 $\times$ 2-stage} & \multirow{11}{*}{(1 000 000, 1 584 010, 0, 0, 2 144 190)} & \multirow{3}{*}{(1 040 000, 0, 0, 2 230 400, 1 495 600)} & 4 172 400 & \multirow{9}{*}{250 000} & \multirow{3}{*}{250 000} & 250 000 & 750 000 \\ 
 & & & 4 759 800 & & & 289 930 & 789 930 \\ 
 & & & 5 172 400 & & & 250 000 & 750 000 \\
 & & & & & & & \\
 & & \multirow{3}{*}{(1 160 000, 1 766 200, 0, 0, 2 213 800)} & 4 668 200 & & \multirow{3}{*}{250 000} & 782 250 & 1 282 200 \\
 & & & 5 092 100 & & & 761 110 & 1 312 180 \\
 & & & 5 667 500 & & & 500 000 & 1 000 000 \\
  & & & & & & & \\
 & & \multirow{3}{*}{(5 037 300, 0, 0, 0, 0)} & 4 653 700 & & \multirow{3}{*}{500 000} & 1 177 800 & 1 927 800 \\
 & & & 5 366 900 & & & 677 300 & 1 427 300 \\
 & & & 7 305 600 & & & 250 000 & 1 000 000 \\ \hline \hline

\end{tabular}
}
\end{center}
\caption{Result of three-stage and two-stage moving horizon models.}\label{tab:3s}
\end{table}
\end{landscape}

For better comparison and tracking of the trade-offs between objective functions in different scenarios, we visualised the objective functions' values of the solutions generated by these two approaches by the scenario-based empirical achievement functions and scenario-based heatmaps \citep{Shavazipour2021b}, particularly proposed for multi-scenario multi-objective decision-making problems.

Using so-called \emph{all-in-one SB-EAFs}, Figure \ref{fig:sbeaf} compares the two objective function values in all nine scenario paths. The red broken line (\textcolor{red}{\scriptsize{\sout{$\circ$}}}) demonstrates the objective function values of the solution generated by the three-stage model, while the blue one (\textcolor{blue}{\tiny{\sout{$\square$}}}) represents the objective function values of the solution generated by the two-stage moving horizon models. 
$s_1, s_2, \dots, s_9$, are corresponds to the scenario paths. The colour code (shown on the right-hand side of each plot) describes the number of scenarios in which that region is achievable by at least one solution.
For example, the yellow area (\cbox{eafbest}) demonstrates the values can be attained only by one solution in one scenario called the \emph{best possible attainment surface}, e.g., in the best case, the highest amount of money can be reached if the solution generated by the 3-stage model (\textcolor{red}{\tiny{\sout{$\circ$}}}) is chosen and scenario $S_9$ unfolded. On the other hand, the highest values for the second objective function (withdrawal) can be seized if the solution generated by the two-stage moving horizon models (\textcolor{blue}{\tiny{\sout{$\square$}}}) are picked and scenario $S_7$ revealed. 
In contrast, the dark purple area (\cbox{eafworst}) describes the \emph{worst attainment surface}---the area that guarantees the objective values achievable in all nine scenarios by at least one solution—for instance, choosing the solution generated by the two-stage moving horizon models (\textcolor{blue}{\scriptsize{\sout{$\square$}}}) assures the decision-maker of having at least €$4$ $172$ $400$ at the end of the investment period, although the withdrawal is minimal. 
The trade-offs between objectives in various scenarios can also be tracked in this figure.
Figure \ref{fig:sbheatmaps} helps to get the exact objective values in all selected scenarios and a more detailed trade-off analysis from a different perspective. 

Moreover, as seen in these figures, in most scenarios, the three-stage model provides better achievements for the second objective (total withdrawals). In comparison, the two-stage model contains better achievements for the first objective (available funds). The complex concept of Pareto optimality in the multi-scenario multi-objective optimisation context and existing trade-offs between Optimality in a single scenario and robustness over a broader range of scenarios \citep{Shavazipour2021b} make the comparison challenging. Furthermore, the different time windows in the three-stage and two-stage models may influence the importance weights in each meta-objectives leading to different Pareto optimal solutions. Indeed, it is possible to generate different Pareto optimal solutions (based on the definition of Pareto optimality in multi-scenario multi-objective optimisation in Section 2.1) by using each model. 
However, besides the apparent trade-offs between the objective functions in the first eight scenario paths, the solution generated by the two-stage moving horizon models is dominated in the last scenario ($s_9$) by the solution generated by the three-stage model.
This confirms our claim on the possibility of generating dominated (or infeasible) solutions in some scenarios by the two-stage moving horizon models raised in the previous section. The reason is that the three-stage model is looking further ahead than the two-stage model, considers the consequences of the initial decision, and is expected to find better and more robust solutions than the two-stage model.

\begin{figure}[!ht]
    \centering
    \includegraphics[width=0.7\textwidth]{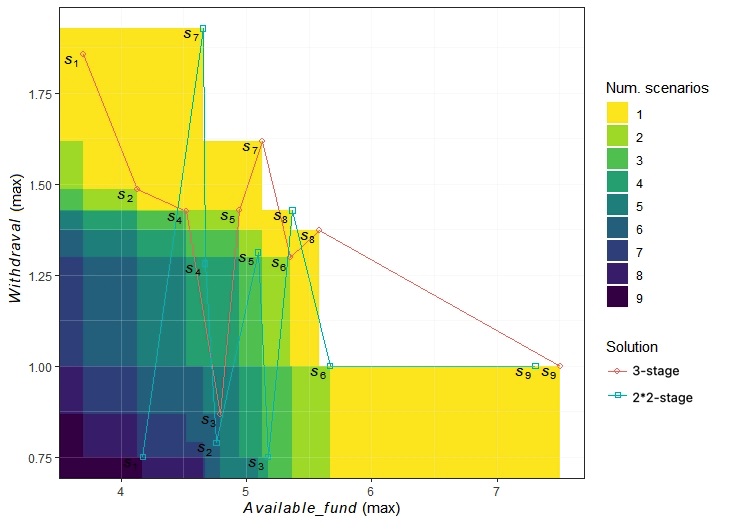}
    
    \caption{Comparing the objective function values attainable by the solutions generated by the three-stage model (\textcolor{red}{\scriptsize{\sout{$\circ$}}}) and two-stage moving horizon models (\textcolor{blue}{\tiny{\sout{$\square$}}}) in all nine scenario paths using the
    all-in-one SB-EAF visualisation.
    Points connected by a line denote a solution evaluated in different scenarios. Coloured areas show regions of the objective space that can be attained within a particular number of scenarios by a solution.}
    \label{fig:sbeaf}
\end{figure}

\begin{figure}
    \centering
    
    \includegraphics[width=0.7\textwidth]{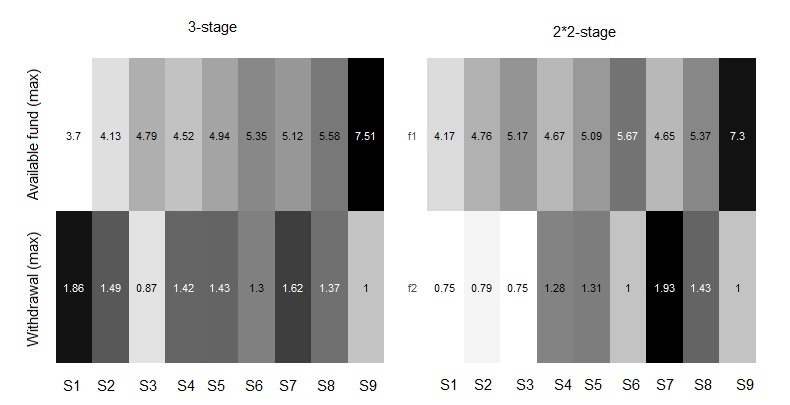}
    \caption{Comparing two solutions generated by the three-stage model and two-stage moving horizon models in nine scenarios via Heatmaps. Darker values are better (both objectives are to be maximised.}
    \label{fig:sbheatmaps}
\end{figure}

\section{Robustness analysis and further discussion}\label{dis}
As seen in the last section, in our example, the solution generated by the three-stage model dominated the solution produced by the two-stage moving horizon in one scenario. However, it was the best-case scenario, and the objective values in other scenarios show two compromise solutions. Under deep uncertainty, we are looking for robust solutions that work satisfactorily in a wider range of scenarios rather than an optimal solution for a specific scenario. Therefore, we also need to compare the generated solutions based on a robustness measure.

To check the solutions' robustness, in this study, we consider the total profit over the given horizon (which is always an essential hidden goal of all decision-makers) as the robustness measure meaning that if a portfolio solution meets the minimum acceptable profit set by the decision-maker in all scenario-paths, it will be considered as a robust portfolio solution. For other popular robustness measures used in decision-making under deep uncertainty, we refer the reader to, e.g., \cite{mcphail2018robustness, Shavazipour2021a}.
Table \ref{tab:benefit23s} evaluates the profit that could be obtained by implementing the solutions generated by each approach after two investment periods. The profit on investments is computed by subtracting the summation of all withdrawals and remaining money in investments at the end of the horizon from our initial capital (i.e., 5 million euros). Here, the decision-maker sets the minimum acceptable profit equal to $10\%$ (i.e., €$500$ $000$). 
Considering the given robustness measure, the solution generated by the three-stage model is robust and guarantees the minimum profit of €$500$ $000$ in all scenarios. In comparison, the solution generated by the two-stage moving horizon models not only failed to reach the robustness measure in the worst-case scenario, but it also gave rise to a €$77$ $640$ loss in the worst-case scenario. Therefore, as expected, the solution generated by the three-stage model is more robust than the solution generated by the other approach.

Surely, looking at more stages must improve the solution. However, does the robustness of the T-stage structure have no cost, or is it always parsimonious? Is getting more robust solutions utilising a T-stage model good enough to warrant the added calculations? In our example, the computation time of generating a solution with both approaches was a few seconds, therefore, using the three-stage model to ensure generating a more robust solution is efficient and reasonable. However, we must warn that, adding more stages would exponentially increase the size of the optimisation model. This would be a major limitation to appending more stages to the model, particularly in complex large-scale problems with many scenarios, and in practice, we may not be able to run a model efficiently with more than two or three stages.  
Then, proving the efficiency of using the T-stage model rather than some successive two-stage models is necessary and needs to be evaluated case by case.

\begin{table}[!h]
\begin{center}
{
\begin{tabular}{|cc|c|c|c|c|c|c|} \hline
   \multicolumn{2}{|c|}{Structure} &                           \multicolumn{3}{|c|}{3-stage} & \multicolumn{3}{|c|}{2$\times$ 2-stage}  \\ \cline{1-8}
 \multicolumn{2}{|c|}{Scenario-path} & Remained fund & Withdrawal & Profit & Remained fund & Withdrawal & Profit \\ \hline
 \multirow{3}{*}{$k(1)=S_2$} & $k(2)=S_1$ & 3 697 200 & 1 857 300   & 554 600     & 4 172 400 & 750 000   & \textcolor{red}{-77 640}   \\ 
                             & $k(2)=S_2$ & 4 127 600 & 1 485 600   & 613 200     & 4 759 800 & 789 930   & 549 710   \\ 
                             & $k(2)=S_3$ & 4 791 900 & 866 700 & 658 500     & 5 172 400 & 750 000 & 922 360   \\
 & & & & & & & \\
 \multirow{3}{*}{$k(1)=S_3$} & $k(2)=S_2$ & 4 521 000 & 1 424 800 & 945 800     & 4 668 200 & 1 282 200   & 950 500  \\ 
                             & $k(2)=S_3$ & 4 944 700 & 1 428 000 & 1 372 700     & 5 092 100 & 1 312 180 & 1 353 200  \\ 
                             & $k(2)=S_4$ & 5 352 800 & 1 297 900 & 1 650 700     & 5 667 500 & 1 000 000 & 1 667 500   \\
  & & & & & & & \\
 \multirow{3}{*}{$k(1)=S_4$} & $k(2)=S_3$ & 5 123 500 & 1 617 600 & 1 741 100     & 4 653 700 & 1 927 800 & 1 581 600    \\ 
                             & $k(2)=S_4$ & 5 583 700 & 1 372 300 & 1 956 000     & 5 366 900 & 1 427 300  & 1 794 200    \\ 
                             & $k(2)=S_5$ & 7 507 800 & 1 000 000 & 3 507 800     & 7 305 600 & 1 000 000 & 3 305 600   \\ \hline \hline

\end{tabular}
}
\end{center}
\caption{Robustness comparison of generated solutions by the three-stage and two-stage moving horizon models based on total profit. }\label{tab:benefit23s}
\end{table}

\section{Conclusions}\label{3sconclusion}

As our world gets distant from what is called a normal situation and becomes more complex, predicting the exact future state of the world seems impossible, particularly in long-term horizons, which makes decision-making extremely challenging. Therefore, finding robust and resilient solutions that perform satisfactorily in many plausible scenarios is getting more attention from many decision-makers rather than finding the optimal solution for an exact state. On the other hand, robust solutions identify with conventional methods, which mainly consider the worst-case scenario, known as too costly and conservatism solutions. Furthermore, in our continuous fast-changing world, long-term planning requires frequent monitoring and adaptations to keep growing and avoid failures. Therefore, to the best of our knowledge, this study presents the first adaptive optimisation framework for multi-criteria robust decision-making under deep uncertainty. Two approaches, the multi-stage multi-scenario multi-objective robust optimisation and the two-stage moving horizon, have been introduced and compared to find dynamic-robust solutions within the proposed framework. The former approach considers the whole planning horizon and generates solutions for multiple stages in one run, while the latter always looks one stage ahead and generates solutions via a two-stage optimisation model. Then, another two-stage model runs after each stage until the entire planning horizon is covered. 

Clearly, since the T-stage model is looking further ahead than the two-stage models in each stage, and the constraints in the T-stage model are simultaneously satisfied, then the optimal solution generated by the T-stage model is no worse than the solutions generated by the two-stage moving horizon models. So, the T-stage model must find a more robust, or at least the same, solution than the two-stage moving horizon models. 
Both approaches have been applied in a case study of sequential portfolio selection with three stages, and their results have been compared, which confirms the fact that the three-stage approach generates more robust solutions. The number of scenario paths in the case study was nine, and corresponding optimisation models were solved quickly. Therefore, using the three-stage model in this case was efficient. 

However, considering more stages exponentially increases the number of scenario paths. Therefore, using an approach with more stages could result in a model that is too complex, computationally expensive and challenging to solve and the efficiency of using such a model must be carefully investigated in each problem. Moreover, in some problems, the degree of uncertainty is so deep that the plausible scenarios in the later stages might change over time, and the set of plausible scenarios must be updated after some stages have passed. In those situations, the two-stage moving horizon approach gives us the opportunity to handle a less computationally (cognitively) expensive problem and/or update the scenario paths as we go into the future. 

\emph{Identifying a way to increase the robustness of the two-stage moving horizon approach}, e.g., by introducing more objective functions or adding some priorities (or weights) for some of the objective functions in some scenarios in the earlier stages, is an interesting study that we leave it as one of our future research directions. Another interesting research direction is to find out \emph{how many stages are worth looking ahead in a continuous planning window when using a moving horizon approach?}. Besides, in this study, we use the decision-maker preferences \emph{a priori} to generate a single Pareto optimal solution. Considering other types of preference incorporation like \emph{a posteriori} or \emph{interactive} methods and conducting a more extensive trade-off and robustness analyses in some real-world applications also lie in the future research directions.

\textcolor{blue}{
\bibliographystyle{agsm}
\bibliography{refs}
}

\section{Appendix}

\subsection{Proof of Theorem~\ref{22sin3s} }\label{proof22sin3s} 
\textbf{Theorem3.2.} Any feasible solution 
to the two-stage moving horizon model 
in a three-stage planning window is a feasible solution for the corresponding three-stage model 
and \textit{vice versa}---i.e., the two-stage moving horizon model is feasible \textbf{\emph{iff}} the three-stage model is feasible. 
\begin{proof}
For simplicity, let us consider the vector form of the models. Model \ref{eq:2svectorform} utilise for the first two-stage model and getting the initial solution($\mathbf{x}^0$). 

\begin{center}
\begin{equation}\label{eq:2svectorform}
 \begin{array}{l l}
Min \quad \mathbf{Z}_{k}= \mathbf{c}^0 \mathbf{x}^0 + (\mathbf{c}_k^0 \mathbf{x}^0 + \mathbf{c}_k^1 \mathbf{x}_k^1);  & k=1,..., p;  \\
 & \\
 s.t.   \mathbf{A}^0 \mathbf{x}^0 \leq \mathbf{b}^0, &  \\
 & \\
   \mathbf{A}_k^0 \mathbf{x}^0 + \mathbf{A}_k^1 \mathbf{x}_k^1 \leq \mathbf{b}_k^1, & k=1,...,p; \\
 & \\
   \mathbf{x}^0, \mathbf{x}_k^1 \geq \mathbf{0}. & k=1,..., p. \\
 & 
\end{array}
\end{equation}
\end{center}

The vector form of the model \ref{eq:22s} as the second two-stage model providing us with the recourse solutions can be rewritten as follows:

\begin{center}
\begin{equation}\label{eq:22svectorform}
 \begin{array}{l l}
Min \quad \mathbf{Z}_{k(1)}= \mathbf{c}_{k(1)}^0 \mathbf{x}_{k(1)}^1 + \mathbf{c}_{k(1)k(2)}^2 \mathbf{x}_{k(1)k(2)}^2;  &  k(2)=1,..., p(2) \\
 s.t.    &  \\
    \mathbf{A}_{k(1)}^1 \mathbf{x}_{k(1)}^1 \leq \mathbf{b}_{k(1)}^1 - \mathbf{A}_{k(1)}^0 \mathbf{x}^{0*}, &  \\
    \mathbf{A}_{k(1)k(2)}^1 \mathbf{x}_{k(1)}^1 + \mathbf{A}_{k(1)k(2)}^2 \mathbf{x}_{k(1)k(2)}^2 \leq \mathbf{b}_{k(1)k(2)}^2 - \mathbf{A}_{k(1)k(2)}^0 \mathbf{x}^{0*}, & k(2)=1,..., p(2), \\
 & \\
   \mathbf{x}_{k(1)}^1, \mathbf{x}_{k(1)k(2)}^2 \geq \mathbf{0}, & k(2)=1,..., p(2). \\
 & 
\end{array}
\end{equation}
\end{center}

Also, the vector form of the three-stage model 
is set down as the following: 
\begin{center}
\begin{equation}\label{eq:3svectorform}
 \begin{array}{l l}
 
 Min \quad \mathbf{Z}_{k(1)}= \mathbf{c'}^0 \mathbf{x}^0 + \mathbf{c'}_{k(1)}^0 \mathbf{x}^0 + \mathbf{c'}_{k(1)}^1 \mathbf{x}_{k(1)}^1  &  k(1)=1,..., p(1); \\
 & \\ 
 
Min \quad \mathbf{Z}_{k(1)k(2)}= \mathbf{c}^0 \mathbf{x}^0 + \mathbf{c}_{k(1)}^0 \mathbf{x}^0 + \mathbf{c}_{k(1)}^1 \mathbf{x}_{k(1)}^1 + \mathbf{c}_{k(1)k(2)}^0 \mathbf{x}^0 & k(1)=1,..., p(1);\\
& \\
   + \mathbf{c}_{k(1)k(2)}^1 \mathbf{x}_{k(1)}^1 + \mathbf{c}_{k(1)k(2)}^2 \mathbf{x}_{k(1)k(2)}^2;  &  k(2)=1,..., p(2);  \\
 & \\
 
 s.t.   \mathbf{A}^0 \mathbf{x}^0 \leq \mathbf{b}^0,(*) &  \\
 & \\

    \mathbf{A}_{k(1)}^0 \mathbf{x}^0 + \mathbf{A}_{k(1)}^1 \mathbf{x}_{k(1)}^1 \leq \mathbf{b}_{k(1)}^1,(**) & k(1)=1,...,p(1); \\ 
 & \\ 
     \mathbf{A}_{k(1)k(2)}^0 \mathbf{x}^0 + \mathbf{A}_{k(1)k(2)}^1 \mathbf{x}_{k(1)}^1 +  \mathbf{A}_{k(1)k(2)}^2 \mathbf{x}_{k(1)k(2)}^2 \leq \mathbf{b}_{k(1)k(2)}^2,(***) & k(1)=1,...,p(1);   \\ 

 & k(2)=1, ...,p(2);\\
   \mathbf{x}^0, \mathbf{x}_{k(1)}^1, \mathbf{x}_{k(1)k(2)}^2 \geq 0, &  k(1)=1,...,p(1); \\
 & k(2)=1, ...,p(2).\\
 &
\end{array}
\end{equation}
\end{center}
 
$\Rightarrow)$ Now, due to the assumption, suppose that the two-stage moving horizon model is feasible, then $ $ $ \forall k=k(1) \in S^1$, and $\forall k(2)\in S^2$, $\exists \underline{X}=(\underline{\mathbf{x}}^0, \underline{\mathbf{x}}_{k(1)}^1)$, and $\exists \overline{\mathbf{X}}=(\overline{\mathbf{x}}_{k(1)}^1, \overline{\mathbf{x}}_{k(1)k(2)}^2), $ $ $ where $\underline{X}$, and $\overline{\mathbf{X}}$, are, respectively, the vectors of feasible solutions for the first(\ref{eq:2svectorform}) and the second(\ref{eq:22svectorform}) two-stage models. Therefore, $\mathbf{X}=(\underline{\mathbf{x}}^0, \overline{\mathbf{x}}_{k(1)}^1, \overline{\mathbf{x}}_{k(1)k(2)}^2)$ is a vector of feasible solutions for the two-stage moving horizon model that satisfies its constraints, then, we have \\

(1) $ $ $ $ $  \underline{\mathbf{x}}^0, \overline{\mathbf{x}}_{k(1)}^1, \overline{\mathbf{x}}_{k(1)}^2  \geq \mathbf{0}, $ $ $ $ \forall k(1), k(2)$; \\

(2) $ $ $ $ $\mathbf{A}^0 \underline{\mathbf{x}}^0 \leq \mathbf{b}^0, \Rightarrow (*)$ is satisfied. \\

(3) $ $ $ $ $ \forall k(1), $ $ \mathbf{A}_{k(1)}^1 \overline{\mathbf{x}}_{k(1)}^1  \leq \mathbf{b}_{k(1)}^1 - \mathbf{A}_{k(1)}^0 \underline{\mathbf{x}}^0 $;\\

$ $

$\Rightarrow \mathbf{A}_{k(1)}^0 \underline{\mathbf{x}}^0 + \mathbf{A}_{k(1)}^1 \overline{\mathbf{x}}_{k(1)}^1  \leq \mathbf{b}_{k(1)}^1 \Rightarrow (**)$, is satisfied. \\

$ $

(4) $ $ $ $ $\forall k(1), k(2), $ $  \mathbf{A}_{k(1)k(2)}^1 \overline{\mathbf{x}}_{k(1)}^1 + \mathbf{A}_{k(1)k(2)}^2 \overline{\mathbf{x}}_{k(1)k(2)}^2   \leq \mathbf{b}_{k(1)k(2)}^2 - \mathbf{A}_{k(1)k(2)}^0 \underline{\mathbf{x}}^0 $;\\

$ $

$\Rightarrow \mathbf{A}_{k(1)k(2)}^0 \underline{\mathbf{x}}^0 + \mathbf{A}_{k(1)k(2)}^1 \overline{\mathbf{x}}_{k(1)}^1 + \mathbf{A}_{k(1)k(2)}^2 \overline{\mathbf{x}}_{k(1)k(2)}^2  \leq \mathbf{b}_{k(1)k(2)}^2 \Rightarrow (***)$, is satisfied. Thus, \\

$ $

$\mathbf{X}$ is a feasible solution for the three-stage model. \\

$ $

$\Leftarrow)$ Suppose that the 3-stage model is feasible, then, \\

$ $

$ $ $ \forall k=k(1) \in S^1$, and $\forall k(2)\in S^2$, $\exists X'=(\widehat{\mathbf{x}}^0, \widehat{\mathbf{x}}_{k(1)}^1, \widehat{\mathbf{x}}_{k(1)k(2)}^2)$, subject to:

$ $

($i$) $ $ $ $ $  \widehat{\mathbf{x}}^0, \widehat{\mathbf{x}}_{k(1)}^1, \widehat{\mathbf{x}}_{k(1)k(2)}^2  \geq \mathbf{0},  $ $ $ $ \forall k(1), k(2) $; \\

($ii$) $ $ $ $ $ \mathbf{A}^0 \widehat{\mathbf{x}}^0 \leq \mathbf{b}^0$;  \\

($iii$) $ $ $ $ $ \forall k(1), $ $ \mathbf{A}_{k(1)}^0 \widehat{\mathbf{x}}^0 + \mathbf{A}_{k(1)}^1 \widehat{\mathbf{x}}_{k(1)}^1  \leq \mathbf{b}_{k(1)}^1;  $\\

($iv$) $ $ $ $ $\forall k(1), k(2), $ $ \mathbf{A}_{k(1)k(2)}^0 \widehat{\mathbf{x}}^0 + \mathbf{A}_{k(1)k(2)}^1 \widehat{\mathbf{x}}_{k(1)}^1 + \mathbf{A}_{k(1)k(2)}^2 \widehat{\mathbf{x}}_{k(1)k(2)}^2   \leq \mathbf{b}_{k(1)k(2)}^2; $\\


$ (i), (ii)$, and $(iii)$, satisfy the constraints of the first two-stage model in \ref{eq:2svectorform}, while $ (i), (iii)$, and $(iv)$, convince the limitations of the model \ref{eq:22svectorform}. Therefore, the two-stage moving horizon model is feasible, and the proof is complete. 

\end{proof}

\subsection{An illustrative example highlighting the robustness of a three-stage model vs. two-stage moving horizon models}\label{app2}

\begin{exmp}\label{exmp1}
Suppose that the following constraints indicate the feasible region for the first two-stage optimisation model of a three-stage planning horizon.

\begin{equation}\label{eq:1st2s}
\begin{array}{l}
x^0 \leq 1, \\
x^0 - x^1 \leq 1, \\
x^0, x^1 \geq 0, \\
\end{array}
\end{equation}

Moreover, the limitations of the second two-stage optimisation model for the appropriate two-stage moving horizon model under conditions of a scenario have been explained in \ref{eq:2nd2s}. 
\begin{equation}\label{eq:2nd2s}
\begin{array}{l}
 x^1 \leq 1- x^{0*}, \\
x^1 + x^2 \leq \dfrac{1}{2} -x^{0*}, \\
x^1, x^2 \geq 0, \\
\end{array}
\end{equation}
where $x^{0*}$ is an optimal solution to the first two-stage optimisation model (model \ref{eq:1st2s}). 

Let consider $\overline{\mathbf{X}}=(\overline{\mathbf{x}}^0, \overline{\mathbf{x}}^1)=(1,0)$, as a feasible solution for \ref{eq:1st2s}. If we substitute $x^{0*}=1$ into \ref{eq:2nd2s}, then no feasible solution can be found. 

$ $

\begin{footnotesize}
$ x^1 \leq 1 - 1 =0 $ $ \rightarrow x^1 \leq 0$  $ $ (*) $ $ $ $ and $ $ $ $ $x^1, x^2 \geq 0$ $ $ (**)\\

(*) and (**) imply that $ x^1=0,$ $ $ $\Rightarrow$ $ $ $x^1 + x^2 = 0 + x^2 \leq \dfrac{1}{2} -1=-\dfrac{1}{2} $ $ \longrightarrow $ $ x^2 \leq -\dfrac{1}{2} $ \contradiction  \textit{``Contradiction!''}.
\end{footnotesize}

$ $

 The feasible region for the corresponding three-stage model can be described by equations in \ref{eq:1st3s}, which includes simultaneous satisfaction of all the constraints in the first and second two-stage optimisation models (models \ref{eq:1st2s} and \ref{eq:2nd2s}). 

\begin{equation}\label{eq:1st3s}
\begin{array}{l}
x^0 \leq 1,  \\
x^0 - x^1 \leq 1,\\
x^0 + x^1 + x^2 \leq \dfrac{1}{2},\\
x^0, x^1, x^2 \geq 0,\\
\end{array}
\end{equation}

$ $

In contrast to the 2 $\times$ two-stage models, too many feasible solutions can be found for the corresponding three-stage model in the same scenario, such as $x^0=x^1=x^2=\dfrac{1}{n}, \forall n>6$. For example, by setting ``$n=\dfrac{1}{8}$'', we have

\begin{footnotesize}
$ $ $ $ $ $ $ \dfrac{1}{8} \leq 1$,\\

$ $ $ $ $ $ $ \dfrac{1}{8} - \dfrac{1}{8}=0 \leq 1$,\\

$ $ $ $ $ $ $ \dfrac{1}{8} + \dfrac{1}{8} + \dfrac{1}{8}=\dfrac{3}{8} \leq \dfrac{1}{2}$,\\

$ $ $ $ $ $ $ \dfrac{1}{8} \geq 0$.\\
\end{footnotesize}

\end{exmp}

\subsection{Two-stage model for sequential portfolio problem in Section \ref{portfolio}}\label{2s_portfolio}

\begin{center}
\begin{equation}\label{eq:e32s}
 \begin{array}{l l}

Max \quad Z_{1k(1)}= \sum_{j=1}^5 \sum_{i=1}^5 x_{ij}^0 + \sum_{i=1}^5 x_{i7k(1)}^1 & \forall k(1);\\
&\\

Max \quad Z_{2k(1)}= \sum_{i=1}^5 (x_{i6}^0 +  x_{i6k(1)}^1)  & \forall k(1);\\
& \\
 s.t.     \\
 
(\textit{Funds balance constraints}) &\\

  \sum_{j=1}^6 (1+p_{ij})x_{ij}^0 = b_i^0, & i=1,...,5; \\
 & \\
 
   (1+p_{i6}) x_{i6k(1)}^1 + x_{i7k(1)}^1 =\sum_{j=1}^5 (1+c_{ik(1)}) x_{ji}^0  & \forall i, k(1); \\ 
 & \\

 (\textit{Minimum Withdrawal constraints})&\\
 
  \sum_{i=1}^5  x_{i6}^0 \geq b_{6}^0, &  \\ 
 &\\
   \sum_{i=1}^5  x_{i6k(1)}^1 \geq b_{6k(1)}^1, & \forall k(1); \\ 
&\\
 
 (\textit{Non-negativity constraints})&\\
 
   x_{ij}^0, x_{i6k(1)}^1, x_{i7k(1)}^1 \geq 0. & i=1,...,5; j=1,...,5; \forall k(1);\\
&  

\end{array}
\end{equation}
\end{center}

\subsection{Three-stage GP model for sequential portfolio problem in Section \ref{portfolio}} \label{3sRGP}

\begin{center}
\begin{equation}\label{eq:e33sGP}
 \begin{array}{l l}
 Min \quad  \psi = \phi  + \epsilon \sum_{k(2)\in K_2} \sum_{k(1)\in K_1} \sum_{n=1}^2 (\ \delta_{nk(1)k(2)})  \\
  & \\
 s.t.     \delta_{nk(1)k(2)} - \phi \leq  0, &      n=1,2, \forall k(1), k(2); \\
 & \\
  Z_{nk(1)k(2)} - \delta_{nk(1)k(2)} = g_{nk(1)k(2)}, &      n=1,2, \forall k(1), k(2); \\

& \\
 
(\textit{Funds balance constraints}) &\\

   \sum_{j=1}^6 (1+p_{ij})x_{ij}^0 = b_i^0, & i=1,...,5; \\
 & \\
  \sum_{j=1}^6 (1+p_{ij}) x_{ijk(1)}^1 = \sum_{j=1}^5 (1+c_{ik(1)}) x_{ji}^0   & i=1,2,3,4,5;\\
 & \forall k(1); \\
 & \\
 
    (1+p_{i6}) x_{i6k(1)k(2)}^2 + x_{i7k(1)k(2)}^2 =\sum_{i=1}^5 (1+c_{jk(1)k(2)}) x_{jik(1)}^1   & \forall i, k(1), k(2); \\ 
 & \\

 (\textit{Minimum Withdrawal constraints})&\\
 
  \sum_{i=1}^5  x_{i6}^0 \geq b_{6}^0, &  \\ 
 &\\
   \sum_{i=1}^5  x_{i6k(1)}^1 \geq b_{6k(1)}^1, & \forall k(1); \\ 
&\\
   \sum_{i=1}^5  x_{i6k(1)k(2)}^0 \geq b_{6k(1)k(2)}^2, & \forall k(1), k(2); \\ 
 & \\

 (\textit{Non-negativity constraints})&\\
 
   x_{ij}^0, x_{ijk(1)}^1, x_{i6k(1)k(2)}^2, x_{i7k(1)k(2)}^2 \geq 0. & i=1,...,5; j=1,...,6; \\
& \forall k(1), k(2). \\
  \phi, \delta_{nk(1)k(2)} \textit{free of sign.} & \forall n, k(1), k(2). \\
 &
\end{array}
\end{equation}
\end{center}

\subsection{Two-stage GP model for sequential portfolio problem in Section \ref{portfolio}}\label{2sRGP}

\begin{center}
\begin{equation}\label{eq:e32sGP}
 \begin{array}{l l}
 Min \quad  \psi = \phi  + \epsilon  \sum_{k(1)\in K_1} \sum_{n=1}^2 ( \delta_{nk(1)})  \\
  & \\
 s.t.    \delta_{nk(1)} - \phi \leq  0, &      n=1,2, k(1)\in K_1; \\
 & \\
  z_{nk(1)} - \delta_{nk(1)} = g_{nk(1)}, &      n=1,2, k(1)\in K_1; \\
& \\
 
(\textit{Funds balance constraints}) &\\

   \sum_{j=1}^6 (1+p_{ij})x_{ij}^0 = b_i^0, & i=1,...,5; \\
 & \\
 
   (1+p_{i6}) x_{i6k(1)}^1 + x_{i7k(1)}^1 =\sum_{j=1}^5 (1+c_{ik(1)}) x_{ji}^0  & \forall i, k(1); \\ 
 & \\

 (\textit{Minimum Withdrawal constraints})&\\
 
  \sum_{i=1}^5  x_{i6}^0 \geq b_{6}^0, &  \\ 
 &\\
   \sum_{i=1}^5  x_{i6k(1)}^1 \geq b_{6k(1)}^1, & \forall k(1); \\ 
&\\
 
 (\textit{Non-negativity constraints})&\\
 
   x_{ij}^0, x_{i6k(1)}^1, x_{i7k(1)}^1 \geq 0. & i=1,...,5; j=1,...,5; \forall k(1);\\
&  \\
  \phi, \delta_{nk(1)}  \textit{free of sign.} & \forall n, k(1). \\
 &
\end{array}
\end{equation}
\end{center}

\subsection{Notations for sequential portfolio problem }\label{app1} 
\begin{table}[htbp!]
\begin{center}
{
\begin{tabular}{|ll|} \toprule
 \textbf{Notation} &\\
$I; I={1,2,3,4,5}$ : & Set of investment options.\\

& \\
\textbf{Decision variables (stage `0')}&   \\ 
 
 $ x_{ij}^0 \in \Re $ : & Amount of fund transferring from investment option $i \in I$ to \\
 & investment option $j \in I$ in stage `0'. \\ 
 
  $ x_{i6}^0 \in \Re $ : & Amount of fund withdrawing from investment option $i \in I$ to \\
  & consumption expenditure in stage `0'. \\ 
\textbf{Decision variables (stage `1')} &   \\ 
  \emph{2-stage model} &\\
 $ x_{i6k(1)}^1 \in \Re $ : & Amount of fund withdrawing from investment option $i \in I$ \\
  & if scenario $S_{k(1)}$ revealed ($S_{k(1)}\in{S_2, S_3, S_4}$). \\ 
 $ x_{i7k(1)}^1 \in \Re $ : & Amount of fund available in investment option $i \in I$ \\
  & if scenario $S_{k(1)}$ revealed ($S_{k(1)}\in{S_2, S_3, S_4}$). \\ 
 \emph{3-stage model} & \\
$ x_{ijk(1)}^1 \in \Re $ : & Amount of fund transferring from investment option $i \in I$ to \\
& investment option $j \in I$ in stage `1',\\
& if state $S_{k(1)}$ revealed ($S_{k(1)}\in{S_2,S_3,S_4}$). \\

\textbf{Decision variables (stage `2')}&   \\

  $ x_{i6k(1)k(2)}^2 \in \Re $ : & Amount of fund withdrawing from investment option $i \in I$ \\
  & if meta-scenario $S_{k(1)k(2)}$ revealed ($S_{k(1)k(2)}\in{S_1, S_2, S_3, S_4, S_5}$). \\ 
 $ x_{i7k(1)k(2)}^2 \in \Re $ : & Amount of fund available in investment option $i \in I$ \\
  & if meta-scenario $S_{k(1)k(2)}$ revealed ($S_{k(1)k(2)}\in{S_1, S_2, S_3, S_4, S_5}$). \\ 
 
 \textbf{Parameters} &\\
 $ c_{jk(1)}$: & The percentage of the growth of investment option $j\in I$, \\
 & if state $S_{k(1)}$ revealed ($S_{k(1)}\in{S_2,S_3,S_4}$).\\
 
 $ c_{jk(1)k(2)}$: & The percentage of the growth of investment option $j\in I$ \\
 & if meta-scenario $S_{k(1)k(2)}$ revealed ($S_{k(1)k(2)}\in{S_1, S_2, S_3, S_4, S_5}$).\\
$ p_{ij} $: & Percentage of loss of funds (penalty cost) for transferring \\
& between each pair of investment$(i,j)$.\\
$ p_{i6} $: & Percentage of loss of funds (penalty cost) for withdrawal money \\
& from investment $i \in I$.\\
&\\  
 
 $ b_i^0 $: & Available funds in investment options $ i \in I $.\\
 &\\
 $ b_6^0 $: & Minimum required fund to withdrawal at stage `0'.\\
 &\\
 $ b_{6k(1)}^1 $: & Minimum required fund to withdrawal at stage `1'.\\
 & if state $S_{k(1)}$ revealed ($S_{k(1)}\in{S_2,S_3,S_4}$).\\
 $ b_{6k(1)k(2)}^1 $: & Minimum required fund to withdrawal at stage `1'.\\
 & if meta-scenario $S_{k(1)k(2)}$ revealed ($S_{k(1)k(2)}\in{S_1, S_2, S_3, S_4, S_5}$).\\
 
 &\\
$g_{nk(1)} \in \Re , $ : & Goal $n$ in scenario $k(1)$ in 2-stage model.\\

$\delta_{nk(1)} \in \Re , $ : & Deviations from the goal $n$ in scenario $k(1)$ in 2-stage model.\\
&\\
$g_{nk(1)k(2)} \in \Re , $ : & Goal $n$ in meta-scenario $k(1)k(2)$ in 3-stage model.\\

$\delta_{nk(1)k(2)} \in \Re , $ : & Deviations from the goal $n$ in meta-scenario $k(1)k(2)$ in 3-stage model.\\

\hline
      
\end{tabular}
}
\end{center}
\caption[Example \ref{e3} variables notation]{Variables notation} \label{tab:var}
\end{table}

\end{document}